\documentclass[11pt]{article}
\usepackage[latin1]{inputenc}
\usepackage{amsmath,amsfonts,amscd,amssymb,amsthm}
\usepackage{color}
\usepackage[]{hyperref}
\usepackage{makeidx}
\newtheorem{thm}{Theorem}[section]
\newtheorem{cor}{Corollary}[section]
\newtheorem{lem}{Lemma}[section]
\newtheorem{prop}{Proposition}[section]
\newtheorem{deft}{Definition}[section]
\newtheorem{rek}{Remark}[section]
\newtheorem{exam}{Example}[section]

\let\noi=\noindent

\def\N{\mathbb{N}} 

\def\init{\mbox{in}}
\makeindex
\title{Key Polynomials}
\author{Wael Mahboub\\
Institut de Math\'ematiques de Toulouse\\
UMR 5219 du CNRS,\\
Universit\'e Paul Sabatier\\
118, route de Narbonne\\
31062 Toulouse cedex 9, France.\\email:
wmahboub@hotmail.com}
\begin{document}
\maketitle

\section{Introduction.}\label{Intro}

The notion of key polynomials (and the related notion of augmented valuations) was first introduced in 1936 by S. Maclane in the case of discrete rank 1 valuations (see \cite{M1}, \cite{M2} and \cite{MS}).

Let $K \rightarrow L$ be a field extension and $\nu$ a
valuation of K. The original motivation for introducing key polynomials was the problem of describing all the extensions $\mu$ of $\nu$ to $L$.

Take a valuation $\mu$ of L extending the valuation $\nu$. In the case when $\nu$ is discrete of rank $1$ and $L$ is a simple algebraic
extension of $K$ Maclane introduced the notions of key polynomials for $\mu$ and augmented valuations and proved that $\mu$ is obtained as a limit of a family of augmented valuations on the polynomial ring $K[x]$ (\cite{M2}, p. 377, Theorem 8.1).

Objects very closely related to key polynomials, called approximate roots, appeared in 1973 in the work of Abhyankar and Moh (\cite{AM1} and \cite{AM2}), and independently in the  Th\`{e}se d'Etat of Monique Lejeune-Jalabert \cite{L}. See also \cite{Spi} for another version of the theory of approximate roots in regular two-dimensional local rings.
More recently, the notion of key polynomials appeared in the work of Spivakovsky and Teissier on the local uniformization theorem (a local version of resolution of singularities) in arbitrary characteristic.

The relation between key polynomials and resolution of singularities (in the special case of singularities of plane curves) can be briefly described as follows. Let $(C,0)$ be a germ of an irreducible plane curve defined by a polynomial $f\in k[x,y]$. Assume that $f$ is of the form
$f(x,y)=x^{d}+a_{d-1}(y)x^{d-1}+...+a_{0}(y)$, that is, $f(x,y)=P(x)$ with $P$ a monic polynomial with coefficients in
$K=k(y)$. If we call $ L $ the fraction field of the local ring $\mathcal O_{C,0} $, $L$ is the simple extension of $K$ defined by adjoining the variable $x$, satisfying the polynomial relation $P$. Finding a resolution of singularities of the germ $(C,0)$ is closely related to finding valuations $\left\{\mu_1,\dots,\mu_r\right\}$ of $ L $ which extend the $y$-adic valuation $\nu$ of $K$. Precisely, resolution of singularities of $(C,0)$ amounts to finding a regular semi-local birational ring extension $\mathcal O'$ of $\mathcal O_{C,0}$. The locallizations of $\mathcal O'$ at its various maximal ideals are exactly the valuation rings $R_{\mu_1},\dots,R_{\mu_r}$. In the case when the germ is analytically irreducible, there exists a unique extension $\mu$ of $\nu$ to $L$; the family $A$ of augmented valuations that determines $\mu$ is finite and the associated family of key polynomials $(\phi_i)_{0\leq i\leq g}$ is the family of approximate roots of $f$ in the sense of Abhyankar -- Moh -- Lejeune-Jalabert. The germs of curves defined by these polynomials have maximal contact with the curve $C$.

In a series of papers \cite{V0}--\cite{V3} M. Vaqui\'e generalized MacLane's notion of key polynomials to the case of arbitrary valuations $\nu$ (that is, valuations which are not necessarily discrete of rank 1). The main definitions from these papers are reproduced below in \S\ref{vaquie}. In the present paper, we will refer to key polynomials in the sense of Vaqui\'e as \textbf{Vaqui\'e key polynomials}. The MacLane--Vaqui\'e approach is axiomatic in the sense that key polynomials are defined in terms of their abstract valuation-theoretic properties rather than by explicit formulae.

In the paper \cite{HOS}, published in the Journal of Algebra in 2007---the same year as \cite{V1}---F.J. Herrera Govantes, M.A. Olalla Acosta and M. Spivakovsky develop their own notion of key polynomials for extensions $(K,\nu)\rightarrow(L,\mu)$ of valued fields, where $\nu$ is of archimedian rank 1 (not necessarily discrete) and give an explicit description of the limit key polynomials (which can be viewed as ``generalized Artin--Schreier polynomials''). These authors give a definition of a \textbf{complete system} $\left\{Q_i\right\}_{i\in\Lambda}$ of key polynomials, indexed by a well ordered set $\Lambda$ of order type at most $\mathbb N$ if $char\ \frac{R_\nu}{m_\nu}=0$ and at most $\omega\times\omega$ if $char\ \frac{R_\nu}{m_\nu}>0$. In the present paper, we will refer to key polynomials in the sense of Herrera--Olalla--Spivakovsky as \textbf{HOS} key polynomials. The definition of a complete system of HOS key polynomials is recalled below (Definition \ref{complete}). In \cite{HOS} HOS key polynomials $\{Q_i\}_{i\in\Lambda}$ are constructed by transfinite recursion in $i$, along with \textbf{truncations} $\nu_i$ of $\mu$. Each $Q_i$ is described by an explicit formula in terms of the previously defined key polynomials $\{Q_j\}_{j<i}$. Some of the main definitions and results from \cite{HOS} are reproduced in \S\ref{HOSkey}.

Although it seemed very plausible that the two notions of key polynomials are equivalent or at least closely related, to this author's knowledge no precise results to this effect exist in the literature. Our purpose in this paper is to clarify the relationship between the two notions of key polynomials already developed in \cite{HOS} and \cite{V0}--\cite{V3}. Our main results, stated and proved in \S\ref{Comparison}, can be summarized as follows:

Let $(K,\nu)\rightarrow(L,\mu)$ be an extension of valued fields with $rk\ \nu=1$. Let $\{Q_i\}_{i\in\Lambda}$ be a complete system of HOS key polynomials and $\{\nu_i\}_{i\in\Lambda}$ the corresponding truncations of $\mu$. Then:

\begin{enumerate}
\item For each $i\in\Lambda$ the polynomial $Q_i$ is a Vaqui\'e key polynomial for the truncation $\nu_{i_0}$, where $i_0=i-1$ in the case when $i$ has an immediate predecessor, and $i_0<i$ is a suitably chosen element of $\Lambda$, described in more detail in \S\ref{HOSkey}, if $i$ is a limit ordinal (this is our Proposition \ref{HOSimpliesVaquie}). As a matter of notation, we write $i=i_0+$, regardless of whether or not $i$ is a limit ordinal.
\item The family $F=\{\nu_{i}\}_i\in\Lambda$ of valuations constructed in \cite{HOS} can be extended to an admitted family for the valuation $\mu$ (this is our Theorem \ref{admitted}).
\end{enumerate}

Conversely, if $F=\{\mu_{i}\}_{i\in\Lambda}$ is an admissible family of valuations of $K[x]$ which is admitted for $\mu$, and $\left\{Q_{i}\right\}_{i\in\Lambda}$ the family of key polynomials in the sense of Vaqui\'e associated to $F$, then every polynomial $Q_i$ in the family can be written recursively in terms of the polynomials $Q_{i'}$ with $i'<i$ (this is our Proposition \ref{VaquieimpliesHOS}). If the family $F$ contains continued subfamilies (see Definition \ref{continued} below), the set $\Lambda$ need not be well ordered. We end \S4 by explaining how replacing $\Lambda$ by a suitable well ordered subset and then suitably modifying the polynomials $Q_i$ (roughly speaking, by subtracting terms of higher value in the sense defined precisely below) results in a complete family of HOS key polynomials.\\

In \S\ref{example_limit} we give an example of a limit key polynomial in the case when $rk\ \nu=1$, $char \ \frac{R_\nu}{m_\nu}>0$ and the valuations $\nu$ and $\mu$ are centered in local noetherian rings with fields of fractions $K$ and $L$, respectively.\\

I thank Mark Spivakovsky for his advice and Olga Kashcheyeva for the correction of the example of the last section.\\
I thank the referee for useful comments and suggestions, which led to a major rewriting of the
paper.\\

\textbf{Notation}: We will use the notation $\mathbb N$ for the set of strictly positive integers and $\mathbb{N}_0$ for the set of non-negative integers.

\section{Vaqui\'e key polynomials: the definitions.}\label{vaquie}

\begin{deft}
	Let $\nu: K^{*}\rightarrow \Gamma$ be a valuation. Let $(R_{\nu}, M_{\nu}, k_{\nu})$ denote the valuation ring of $\nu$. For $\beta\in\Gamma$, consider the following $R_{\nu}$-submodules of $K$:
	
	\[ P_{\beta}=\left\{y\in K^{*}\ |\ \nu(y)\geq\beta\right\}\cup\left\{0\right\}
\]
	\[ P_{\beta+}=\left\{y\in K^{*}\ |\ \nu(y)>\beta\right\}\cup\left\{0\right\}
\]

 We define
	\[ G_{\nu}=\bigoplus\limits_{\beta\in\Gamma} \frac{P_{\beta}}{P_{\beta+}}
\]

The $k_{\nu}$-algebra $G_{\nu}$ is an integral domain. For any element $y\in K^{*}$ with $\nu(y)=\beta$, the natural image of $y$ in $\frac{P_{\beta}}{P_{\beta+}}\subset G_{\nu}$ is a homogeneous element of $G_{\nu}$ of degree $\beta$, which we will denote by $in_{\nu}y$.
\end{deft}
Let $(K,\nu)$ be a valued field, $x$ an independent variable, and let $\mu$ be a valuation of $K[x]$, extending $\nu$.
\begin{deft}
For all $f$ and $g$ in $K[x]$:
\begin{enumerate}
	\item We say that $f$ and $g$ are $\mu$-equivalent if
          $in_{\mu} f=in_{\mu}g$.
	\item We say that $g$ $\mu$-divides $f$ if there exists $h\in
          K[x]$ such as $f$ is $\mu$-equivalent to $h.g$.
\end{enumerate}	
\end{deft}

\begin{deft}
\begin{enumerate}
	\item We say that a polynomial $\phi$ in $K[x]$ is $\mu$-minimal if, for all $f$ in
          $K[x]$ we have: $\phi$ $\mu$-divides $f$ $\Rightarrow$
          $\deg_{x}f\geq\deg_{x}\phi$.
	\item We say that $\phi$ is $\mu$-irreducible if for all $f$,
          $g$ in $K[x]$ we have:\\  $\phi$ $\mu$-divides $f.g$
          $\Rightarrow$ $\phi$ $\mu$-divides $f$ or $\phi$
          $\mu$-divides $g$.
\end{enumerate}
\end{deft}

\begin{deft}
	(\cite{V}, page 3442)
	A polynomial $\phi$ in $K[x]$ is said to be a \textbf{Vaqui\'e key polynomial} for the valuation $\mu$ of $K[x]$ if $\phi$ satisfies:
\begin{enumerate}
	\item $\phi$ is $\mu$-minimal.
	\item $\phi$ is $\mu$-irreducible.
	\item $\phi$ is monic.
\end{enumerate}
\end{deft}

\begin{exam}
Let $k$ be a field and $k(x,y)$ an extension of $k$, where $x$ and $y$ are two elements, algebraically independent over $k$. Put $K=k(y)$. Let $\nu$ be the $y$-adic valuation on $K$ (in particular, $\nu(y)=1$). Define the valuation $\mu$ on $K[x]$ as follows: for all $f=\sum\limits_{j=0}^{m} b_{j}x^j$ in $K[x]$\\ $\mu(f)=\min\limits_{0\leq j\leq m}\left\{\nu(b_{j})+ \frac{3}{2}j\right\}$ (in particular, we have $\frac{3}{2}=\mu(x)$).\\
For any $c\in k$ the polynomial $\phi=x^2+cy^3$ is a Vaqui\'e key polynomial for the valuation $\mu$.
\end{exam}

Let $\mu$ be a valuation of $K[x]$ and $\phi$ a key polynomial for $\mu$.
We note that every polynomial $f$ in $K[x]$ can be written uniquely in the form
	\[
	f=f_{m}\phi^{m}+f_{m-1}\phi^{m-1}+...+f_{0}
\]
with $\deg_{x}f_j<\deg_x{\phi}$ for all $0\leq j\leq m$.

Let $\Gamma'$ be an ordered abelian group containing the value group $\Gamma$ of $\mu$. Take an element $\gamma\in \Gamma'$ satisfying $\gamma>\mu(\phi)$.

\begin{deft}\label{augmvaluation} Define the valuation $\mu'$ by $\mu'(f)=\min\limits_{0\leq j\leq m}\left\{\mu(f_{j})+j\gamma\right\}$.
We call the valuation $\mu'$ defined by the valuation $\mu$, the key polynomial $\phi$ and the element $\gamma$ an \textbf{augmented valuation} and we denote: $\mu'=[\mu;\mu'(\phi)=\gamma]$.
\end{deft}
\begin{exam}
Keep the notation and hypotheses of Example 2.1. Take $c=1$, so that $\phi=x^2+y^3$. We see that $\mu(\phi)=\min\left\{\frac{3}{2}*2, \nu(y^3)\right\}=3$.\\
Put $\gamma=\frac{10}{3}$ and define the augmented valuation $\mu'$ on $K[x]$ as follows: \\
for every polynomial $f=f_{m}\phi^{m}+f_{m-1}\phi^{m-1}+...+f_{0}$ of $K[x]$,\\
 $\mu'(f)=\min\limits_{0\leq j\leq m}\left\{\mu(f_j)+j\gamma\right\}$ where $\gamma=\mu'(\phi)=\frac{10}{3}>\mu(\phi)=3$.
\end{exam}

\noindent\textbf{Notation.} In the situation of Definition \ref{augmvaluation}, we will sometimes write $[\mu;\mu'(\phi)=\gamma]$ instead of $\mu'$, to emphasize the dependence of $\mu'$ on $\phi$ and $\gamma$.
\begin{deft}\label{augmentediterated}
(\cite{V}, page 3463).
A family $\{\mu_{\alpha}\}_{\alpha\in A}$ of valuations of $K[x]$,
indexed by a totally ordered set $A$, is called a \textbf{family of augmented
iterated valuations} if for all $\alpha$ in $A$, except $\alpha$ the
smallest element of $A$, there exists $\theta$ in $A$, $\theta <
\alpha$, such that the valuation $\mu_{\alpha}$ is an augmented
valuation of the form $\mu_{\alpha}=[\mu_{\theta};
\mu_{\alpha}(\phi_{\alpha})=\gamma_{\alpha}]$, and if we have the
following properties:
\begin{enumerate}
	\item If $\alpha$ admits an immediate predecessor in $A$,
          $\theta$ is that predecessor, and in the case when $\theta$
          is not the smallest element of  $A$, the polynomials
          $\phi_{\alpha}$ and $\phi_{\theta}$ are not
          $\mu_{\theta}$-equivalent and satisfy
          $\deg\phi_{\theta}\leq\deg\phi_{\alpha}$;
	\item if $\alpha$ does not have an immediate predecessor in $A$,
          for all $\beta$ in $A$ such that $\theta<\beta<\alpha$, the
          valuations $\mu_{\beta}$ and $\mu_{\alpha}$ are equal to the
          augmented valuations, respectively,
$$
\mu_{\beta}=[\mu_{\theta};\mu_{\beta}(\phi_{\beta})=\gamma_{\beta}]
$$
and
$$
\mu_{\alpha}=[\mu_{\beta};\mu_{\alpha}(\phi_{\alpha})=\gamma_{\alpha}],
$$
and the polynomials $\phi_{\alpha}$ and $\phi_{\beta}$ have the same degree.
\end{enumerate}
\end{deft}

\begin{deft}\label{continued}
(\cite{V}, page 3464)
	A family of augmented iterated valuations
        $\{\mu_{\alpha}\}_{\alpha\in A}$ is said to be \textbf{continued} if there
        exists a valuation $\mu$ on $K[x]$, an infinite subset
        $\Lambda=\left\{ \gamma_{\alpha}\ |\ \alpha\in A \right\}$ of
        the group $\Gamma$ not containing a maximal element, a family of
        polynomials $\{\phi_{\alpha}\}_{\alpha\in A}$ of the same degree
        $d$, each polynomial $\phi_{\alpha}$ being a key polynomial
        for the valuation $\mu$ with $\mu_{\alpha}=[\mu;
        \mu_{\alpha}(\phi_{\alpha})=\gamma_{\alpha}]$ for all $\alpha$
        in $A$.
\end{deft}
\begin{rek} (\cite{V}, 3463, beginning of \S1.4) Consider a continued family of augmented iterated valuations
        $\{\mu_{\alpha}\}_{\alpha\in A}$. Then all the valuations $\mu_{\alpha}$ have the same value group. In what follows, we will denote this common value group by $\Gamma_\bullet$.
\end{rek}

\begin{deft}
(\cite{V}, page 3464)
	A continued family of augmented iterated valuations
        $\{\mu_{\alpha}\}_{\alpha\in A}$ is said to be \textbf{exhaustive} if the
        set $\Lambda$ satisfies:
	\[ \forall\alpha<\beta\in A, \forall \gamma\in\Gamma, \gamma_{\alpha}<\gamma<\gamma_{\beta} \Rightarrow \gamma\in\Lambda.
\]
\end{deft}
Consider a continued family of augmented iterated valuations
        $F=\{\mu_{\alpha}\}_{\alpha\in A}$ as above, not necessarily exhaustive. Following \cite{V}, page 3455, let $\Phi_\bullet$ denote the set of monic polynomials $\phi$ in $K[x]$, of degree $d$, such that there exist $\alpha,\beta\in A$, depending on $\phi$, satisfying $\mu_\alpha(\phi)<\mu_\beta(\phi)=\mu_{\beta'}(\phi)$ for all $\beta'\in A$ with $\beta'\ge\beta$. Let $\Lambda_\bullet=\left\{\left.\mu_\beta(\phi)\ \right|\ \phi\in\Phi_\bullet,\beta\in A\text{ sufficiently large}\right\}$. Let $Exh(A)$ denote a totally ordered index set, with a fixed order isomorphism $\gamma:Exh(A)\rightarrow\Phi_\bullet$. We will write $\gamma_\alpha$ for $\gamma(\alpha)$ (the only reason we introduce an extra index set at this point is to make the notation consistent with that of \cite{V}). For each $\alpha\in Exh(A)$ pick and fix a polynomial $\phi_\alpha\in\Phi_\bullet$ such that $\mu_{\beta}(\phi_\alpha)=\gamma_\alpha$ for all $\beta\in A$ sufficiently large; by definition of $\Phi_\bullet$ there exists $\alpha_0\in A$ such that
\begin{equation}
\mu_{\alpha_0}(\phi_\alpha)<\gamma_\alpha.\label{eq:alpha0}
\end{equation}
Let $\mu_\alpha=[\mu_{\alpha_0};\mu_\alpha(\phi_\alpha)=\gamma_\alpha]$; the valuation $\mu_\alpha$ does not depend on the choice of $\alpha_0$ satisfying (\ref{eq:alpha0}).

The discussion preceding Lemme 1.17 (\cite{V}, page 3455) shows that the resulting family $Exh(F):=\{\mu_{\alpha}\}_{\alpha\in Exh(A)}$ is a continued family of augmented iterated valuations.
\begin{prop}\label{exhaustive}(\cite{V}, page 3455, lemme 1.17)
	Consider a continued family of augmented iterated valuations
        $\{\mu_{\alpha}\}_{\alpha\in Exh(A)}$ described above. We have the following results:

\begin{enumerate}
	\item All the valuations $\mu_{\alpha}$, $\alpha\in Exh(A)$ have the same value group $\Gamma_\bullet$.
	\item For all $\alpha$ in $Exh(A)$, the interval $]\gamma, \gamma_{\alpha} ]= \left\{\delta\in\Gamma_\bullet\ |\ \gamma<\delta\le\gamma_{\alpha}\right\}$ is contained in $\Lambda_\bullet$.
\end{enumerate}
In particular, the continued family $Exh(F)=\{\mu_{\alpha}\}_{\alpha\in Exh(A)}$ of augmented iterated valuations is exhaustive.
\end{prop}

\begin{cor}\label{exhaustion}
Proposition \ref{exhaustive} shows that for any continued family $F=\left\{\mu_{\alpha}\right\}_{\alpha\in A}$ of augmented iterated valuations we can always add new valuations to the family in order to make the resulting family $Exh(F)$ exhaustive.
\end{cor}

Let $f$ and $g$ be two polynomials of $K[x]$. We say that $f$ $A$-divides $g$ or that $g$ is $A$-divisible by $f$, if there exists $\alpha_{0}\in A$ such that $f$ $\mu_{\alpha}$-divides $g$ for all $\alpha\in A$ with $\alpha>\alpha_{0}$.\\

\begin{deft}
(\cite{V}, page 3465)
	A polynomial $\phi$ of $K[x]$ is said to be a \textbf{limit key polynomial} for the family of valuations $\{\mu_{\alpha}\}_{\alpha\in A}$ if $\phi$ has the following properties:
\begin{itemize}
	\item $\phi$ is monic.
	\item $\phi$ is $A$-minimal, that is to say any polynomial $f$ $A$-divisible by $\phi$ is of degree greater than or equal to $\phi$
	\item $\phi$ is $A$-irreducible, that is to say: for all $f$, $g$ in $K[x]$, if $\phi$ $A$-divides $fg$, then $\phi$ $A$-divides $f$ or $\phi$ $A$-divides $g$.
\end{itemize}
\end{deft}

\begin{exam} For an example of a limit key polynomial, we refer the reader to \cite{V}, page 3478. Another example is given at the end of this paper. One advantage of our example over that of \cite{V} is that it features valuations $\nu$ and $\mu$ centered in a local noetherian ring.
\end{exam}

Now take a family $\{\mu_{\alpha}\}_{\alpha\in A}$ of augmented iterated valuations.\\

\begin{rek}\label{minimaldegreelimit}It can be proved that every monic polynomial $\phi$ satisfying $\mu_{\alpha}(\phi)<\mu_{\beta}(\phi)$ for all $\alpha<\beta$ in $A$, and with a minimal degree among those satisfying this inequality, is a limit key polynomial for the family (\cite{V}, page 3465, Proposition 1.21).
\end{rek}
\medskip

We want to define a valuation $\mu'$ of $K[x]$ starting from the family of augmented iterated valuations $\{\mu_{\alpha}\}_{\alpha\in A}$, from a limit key polynomial $\phi$ for the family $\{\mu_{\alpha}\}_{\alpha\in A}$, and from a value $\lambda$ in $\Gamma'$ that satisfies $\lambda>\mu_{\alpha}(\phi)$ for all $\alpha\in A$.
\\
Consider an $f$ in $K[x]$ such that: there exists $\alpha_{0}\in A$ with $\mu_{\alpha}(f)$ constant for all $\alpha\in A$ such that $\alpha\geq\alpha_{0}$.
We denote
	\[\mu_{A}(f)=\mu_{\alpha_{0}}(f)=\sup\{\mu_{\alpha}(f)\ |\ \alpha\in A\}.
\]

Put $f=f_{m}\phi^{m}+f_{m-1}\phi^{m-1}+...+f_{0}$, then define $\mu'$ as:
	\[ \mu'(f)=\inf\left\{\mu_{A}(f_{j})+j\lambda; 0\leq j\leq m\right\}.
\]
As $\deg(f_{j})<\deg(\phi)$ for all $0\leq j\leq m$ then $\mu_{A}(f_{j})$ is well defined for all $0\leq j\leq m$.
\begin{deft}
	We call the valuation $\mu'$ defined above the \textbf{limit augmented valuation} for the family $\{\mu_{\alpha}\}_{\alpha\in A}$. We denote $\mu'=[(\mu_{\alpha})_{\alpha\in A};\mu'(\phi)=\gamma]$.
\end{deft}

\begin{deft}
(\cite{V}, page 3471).
	A family $S$ of augmented iterated valuations is said to be a \textbf{simple admissible family} if it has the form $S=\{\mu_{i}\}_{i\in I}$, where the set of indices $I$ is the disjoint union $I=B\cup A$, with $B\subset\N$ and $A$ a totally ordered set possibly empty, where the total order on the set $I$ is defined by $i<\alpha$ for all $i$ in $B$ and for all $\alpha$ in $A$, and the following properties hold:
	
\begin{itemize}
	\item For $i\in B$, $i\geq 2$, we have the inequality  $\deg\phi_{i}>\deg\phi_{i-1}$.
	\item If $A\neq \emptyset$, then $B$ is finite, $B=\left\{ 1,...,n\right\}$ and for $\alpha\in A$, we have $\deg \phi_{\alpha}=\deg\phi_{n}$, and the family $\{\mu_{\alpha}\}_{\alpha\in A}$ is an exhaustive family of augmented iterated valuations.
\end{itemize}
	If the set $A$ is empty, i.e if the set of indices is a subset $I$ of $\N$, we say that the family $S=\{\mu_{i}\}_{i\in I}$ is a simple discrete admissible family.
\end{deft}

\begin{deft}
(\cite{V}, page 3472).
	A family of valuations $S=\{\mu_{i}\}_{i\in I}$ is said to be \textbf{admissible} if it is a union of a finite set or a countable set of admissible simple families $S^{(t)}=\left\{\mu_{i}^{(t)}\right\}_{i\in I^{(t)}}$, with $1\leq t<N$ where $N \in\N\bigcup \left\{+\infty\right\}$, and $I^{(t)}=\left\{ 1^{(t)},...,n^{(t)}\right\}\bigcup A^{(t)}$, satisfying:	
\begin{itemize}
	\item All the simple admissible families $S^{(t)}=\left\{\mu_{i}^{(t)}\right\}_{i\in I^{(t)}}$, except possibly the last one in the case $N < + \infty$, are non-discrete simple admissible families.
	\item The first valuation of the family, i.e. the first valuation $\mu_{1}^{(1)}$ of the first simple admissible family $S^{(1)}$, is an augmented  valuation of the valuation $\nu$ of the field $K$ constructed with the key polynomial $\phi_{1}^{(1)}$ of degree one.
	\item For $t\geq 2$, the first valuation $\mu_{1}^{(t)}$ of the simple admissible family $S^{(t)}$ is the limit augmented valuation for the family of valuations $\left\{\mu_{\alpha}^{(t-1)}\right\}_{\alpha\in A^{(t-1)}}$.
\end{itemize}
\end{deft}
\begin{rek}
 The set of indices $I=\bigcup\limits_{t=1}^{N} I^{(t)}$ is totally ordered by $i<j$ for all $i\in I^{(t)}$ and $j\in I^{(s)}$, if $t<s$.
\end{rek}

\begin{prop}\label{admissible_order}
(\cite{V}, page 3472)
 For all the polynomials $f$ in $K[x]$ the family $\{\mu_{i}(f)\}_{i\in I}$ is increasing, which means that for all $i<j$ in $I$, we have $\mu_{i}(f)\leq\mu_{j}(f)$.\\
 Furthermore, if there exists $i<j$ in $I$ such that $\mu_{i}(f)=\mu_{j}(f)$, then for all $k\geq i$, we have also the equality $\mu_{i}(f)=\mu_{k}(f)$.

\end{prop}

\begin{deft}\label{def_admitted}
(\cite{V}, page 3473)
	An admissible family of valuations $F=\{\mu_{j}\}_{j\in I}$ of $K[x]$ is said to be an \textbf{admitted family} for the valuation $\mu$ if it has the following properties:
\begin{itemize}
	\item	For all $j$ in $I$ and all $f$ in $K[x]$, $\mu_{j}(f)\leq\mu(f)$, and we have the equality $\mu_{j}(f)=\mu(f)$ for $f$ of degree strictly less than the degree of the key polynomial $\phi_{j}$ defining the valuation $\mu_{j}$.
	\item If $S^{(t)}=\left\{\mu_{i}^{(t)}\right\}_{i\in I^{(t)}}$ is a non-discrete simple admissible family contained in $F$,
$$
I^{(t)}=B^{(t)}\bigcup A^{(t)},
$$
then for all $\theta\in A^{(t)}$ we have the equality of sets\\
$\left\{\mu_{\alpha}(\phi_{\alpha})\ |\ \alpha\in A^{(t)}, \alpha>\theta\right\}=\left\{\mu(\phi)\ |\ \phi\ \text{monic}, \deg(\phi)=\deg(\phi_{\theta}), \mu_{\theta}(\phi)<\mu(\phi) \right\}$.
\end{itemize}	
\end{deft}

\begin{rek}
From the second condition of the definition above, we notice that the Vaqui\'e limit key polynomial $\phi_{1^{(t+1)}}$ has degree on $x$ strictly greater than the degree of any polynomial $\phi_{\alpha}$, with $\alpha\in A^{(t)}$.
\end{rek}

\begin{deft}
	We say that an admitted family $F=\{\mu_{j}\}_{j\in I}$ for the valuation $\mu$ \textbf{converges to $\mu$} if for all $f\in K[x]$ there exists a $j\in I$ such that $\mu(f)=\mu_j(f)$, which is equivalent to saying that for all $f\in K[x]$ we have:
$$
	\mu(f)=\lim\limits_{j}\mu_{j}(f)=Max\left\{\mu_{j}(f), j\in I\right\}.
$$
\end{deft}

\begin{thm}
(\cite{V}, page 3475)
For all valuation $\mu$ of $K[x]$ extending a valuation $\nu$ of $K$, there exists an admissible family of valuations $(\mu_{i})_{i\in I}$ that converges to $\mu$.
\end{thm}

We recall the statement of lemme 1.4 in \cite{V}.

\begin{lem}\label{lemme1.4}(\cite{V}, page 3443)
Let $\mu'$ be the valuation defined by the valuation $\mu$, the key polynomial $\phi$, and the value $\gamma \in \Gamma$, i.e $\mu'=[\mu;\mu'(\phi)=\gamma]$. Then for all $f$ in $K[x]$ satisfying $\mu(f)=\mu'(f)$, we have:
\begin{enumerate}
\item there exists $h$ in $K[x]$ with $\deg h<\deg\phi$ such that $in_{\mu'}f=in_{\mu'}h$.
\item there exists $g$ in $K[x]$ with $\mu'(g)=\mu(g)$ such that $in_{\mu'}fg=in_{\mu'}1$.
\end{enumerate}

\end{lem}

We recall the statement of Proposition 1.1 and Proposition 1.2 in \cite{V0}.
\begin{prop}\label{sameaugmentedvaluation}(\cite{V0}, page 397)
Let $\mu$ be a valuation of $K[x]$. Let $\phi_{1}$ and $\phi_{2}$ be two key polynomials for the valuation $\mu$, and let $\gamma_{1}>\mu(\phi_{1})$ and $\gamma_{2}>\mu(\phi_{2})$ be two values of a totaly ordered group containing the ordered group of $\mu$. Then the augmented valuations $\mu_{1}=[\mu;\mu_{1}(\phi_{1})=\gamma_{1}]$ and $\mu_{2}=[\mu;\mu_{2}(\phi_{2})=\gamma_{2}]$ defined by these polynomials and these values are equal if and only if $\gamma_{1}=\gamma_{2}$ and if the polynomials $\phi_{1}$ and $\phi_{2}$ have the same degree and satisfy
$\mu(\phi_{1}-\phi_{2})\geq \gamma_{1}=\gamma_{2}$.\\
In this case, the polynomials $\phi_{1}$ and $\phi_{2}$ are $\mu$-equivalent.
\end{prop}

\begin{prop}\label{samelimitaugmentedvaluation}(\cite{V0}, page 398)
Let $\left\{ \mu_\alpha \right\}_{\alpha\in A}$ be a continued admissible family of valuations of $K[x]$ and let $\psi$ and $\psi'$ be two limit key polynomials for the family $\left\{ \mu_\alpha \right\}_{\alpha\in A}$, then the polynomials $\psi$ and $\psi'$ are $\mu_{\alpha}$-equivalent for all $\alpha$ sufficiently large. Furthermore  the limit augmented valuations $\mu_{1}=[(\mu_{\alpha})_{\alpha \in A};\mu_{1}(\psi)=\gamma]$ and $\mu_{1}'=[(\mu_{\alpha})_{\alpha \in A};\mu_{1}'(\psi')=\gamma']$ defined, respectively, by $\psi$ and $\psi'$ and the values $\gamma$ and $\gamma'$ are equal if and only if $\gamma=\gamma'$ and if the polynomials $\psi$ and $\psi'$ satisfy $\mu_{A}(\psi-\psi')\geq \gamma >\mu_{\alpha}(\psi)=\mu_{\alpha}(\psi')$.\\
\end{prop}

\section{HOS key polynomials: definitions and some basic results.}\label{HOSkey}

The paper \cite{HOS} define a well ordered set $\mathbf Q=\{Q_i\}_{i\in\Lambda}$
 of  key polynomials of a valuation $\mu$ recursively in $i$.
 As the definition of HOS key polynomials is long,
 we can not repeat it all, we refer the reader to the paper \cite{HOS} for a detailed definition. But, in this section we will summarize the main aspects of the definition of HOS key polynomials, and also recall some basic definitions and results from \cite{HOS}.\\

As in the previous section, let $(K,\nu)$ be a valued field, $\Gamma$ the value group of $\nu$, $x$ an independent variable, and let $\mu$ be a valuation of $K[x]$, extending $\nu$, with values in an ordered abelian group $\Gamma'$. For an element $\beta\in\Gamma'$, let $P'_\beta=\{y\in K[x]\ |\ \mu(y)\ge\beta\}$. Let $\Gamma'_{1}$ denote the smallest non-zero isolated subgroup of $\Gamma'$. Assume that $rk\ \nu=1$.\\

\noindent\textbf{Notation.} For an element $l\in\Lambda$, we will denote by $l+1$
the immediate successor of $l$ in $\Lambda$. The immediate predecessor
of $l$, when it exists, will be denoted by $l-1$. For a positive
integer $t$, $l+t$ will denote the immediate successor of
$l+(t-1)$.
\medskip

 We take this opportunity to correct a misprint in the definition of a complete set of HOS key polynomials in \cite{HOS}. The correct definition is:
\begin{deft}\label{complete}(\cite{HOS}, page 1038) A {\bf complete set of HOS key polynomials} for $\mu$ is a well
ordered collection
$$
\mathbf Q=\{Q_i\}_{i\in\Lambda}
$$
of elements of $K[x]$ such that for each $\beta\in\Gamma'$ the additive group $P'_\beta$ is generated
by products of the form\ \ $a\cdot\prod\limits_{j=1}^sQ_{i_j}^{\gamma_j}$, $a\in K$, such that
$\sum\limits_{j=1}^s\gamma_j\mu(Q_{i_j})+\nu(a)\ge\beta$.
\end{deft}
Note, in particular, that if $\mathbf Q$ is a complete set of HOS key polynomials
then their images $\init_\mu Q_i\in G_\mu$ generate $G_\mu$ as a $G_\nu$-algebra.
\begin{rek} The results of \cite{HOS} are stated and proved for simple \textbf{algebraic} extensions
$$
K\hookrightarrow K(x)
$$
of valued fields, but those cited in the present paper are equally valid (with the same proofs) for simple pure transcendental extensions.
\end{rek}
We look for complete systems of HOS key polynomials such that the order type of the well ordered set $\Lambda$ is the smallest possible (it is shown in \cite{HOS} that this order type is at most $\mathbb N$ is $char\ \frac{R_\nu}{m_\nu}=0$ and at most $\omega\times\omega$ if $char\ \frac{R_\nu}{m_\nu}>0$).
\medskip

\noindent\textbf{Notation.} For $l\in\Lambda$, $\mathbf Q_l$ will stand for $\{Q_i\}_{i<l}$ and $\beta_l$ for $\mu(Q_l)$.
\medskip

The paper \cite{HOS} constructs, recursively in $i$, HOS key polynomials $\{Q_i\}_{i\in\Lambda}$ and strictly positive integers $\{\alpha_{i}\}_{i\in\Lambda}$ such that for each $l\in\Lambda$ all but finitely many of the $\alpha_{i}$ with $i\le l$ are equal to 1. We will describe below the main aspects of this construction.

Take a polynomial $h=\sum\limits_{i=0}^sd_ix^i\in K[x]$, $d_i\in K$.
\begin{deft}(\cite{HOS}, page 1039) The first {\bf Newton polygon} of $h$ with respect to $\nu$ is
the convex hull $\Delta_1(h)$ of the set
$\bigcup\limits_{i=0}^s\left(\left(\nu(d_i),i\right)+
\left(\Gamma_+\oplus\Bbb Q_+\right)\right)$ in $\Gamma\oplus\Bbb Q$.
\end{deft}
To an element $\beta_1\in\Gamma'_+$, we associate the following valuation
$\nu_1$ of $K[x]$: for a polynomial $h=\sum\limits_{i=0}^sd_ix^i$, we put
$$
\nu_1(h)=\min\left\{\left.\nu(d_i)+i\beta_1\ \right|\ 0\le i\le s\right\}.
$$
Consider an element $\beta_1\in\Gamma'_+$.
\begin{deft}(\cite{HOS}, page 1039) We say that $\beta_1$ {\bf determines a side} of $\Delta_1(h)$
if the following condition holds. Let
$$
S_1(h,\beta_1)=\left\{\left.i\in\{0,\dots,s\}\ \right|\
i\beta_1+\nu(d_i)=\nu_1(h)\right\}.
$$
We require that $\#S_1(h,\beta_1)\ge2$.
\end{deft}
Let $\beta_1=\mu(x)$. Then for any $h\in K[x]$ we have
\begin{equation}\label{nu1mu}
\nu_1(h)\le\mu(h)
\end{equation}
by the axioms for valuations. If equality holds in (\ref{nu1mu}) for all $h\in K[x]$, we put $\Lambda=\{1\}$, $x=Q_1$ and stop. The definition of key polynomials is complete. From now on, assume that there exists a polynomial $h\in K[x]$ such that $\nu_1(h)<\mu(h)$.
\begin{prop}(\cite{HOS}, page 1039)
Take a polynomial $h=\sum\limits_{i=0}^sd_ix^i\in
K[x]$ such that
$$
\nu_1(h)<\mu(h).
$$
Then
$$
\sum\limits_{i\in S(h,\beta_1)}\init_\nu d_i \init_{\mu}x^i=0.
$$
\end{prop}

\begin{cor}(\cite{HOS}, page 1040) Take a polynomial $h\in K[x]$ such that $\nu_1(h)<\mu(h)$. Then
$\beta_1$ determines a side of $\Delta_1(h)$.
\end{cor}

\noindent\textbf{Notation.} Let $X$ be a new variable. Take a polynomial
$h$ as above. We denote
$$
\init_1h:=\sum\limits_{i\in S_1(h,\beta_1)} \init_\nu d_iX^i.
$$
The polynomial $\init_1h$ is quasi-homogeneous in $G_\nu[X]$, where the
weight assigned to $X$ is $\beta_1$. Let
\begin{equation}\label{inihfact}
\init_1h=v\prod\limits_{j=1}^tg_j^{\gamma_j}
\end{equation}
be the factorization of $\init_1h$ into irreducible factors in $G_\nu[X]$. Here $v\in
G_\nu$ and the $g_j$ are monic polynomials in $G_\nu[X]$ (to be precise, we first
factor $\init_1h$ over the field of fractions of $G_\nu$ and then observe that all
the factors are quasi-homogeneous and therefore lie in $G_\nu[X]$).
\begin{prop}(\cite{HOS}, page 1040)
\begin{enumerate}
\item The element $\init_{\mu}x$ is integral over $G_\nu$.

\item The minimal polynomial of $\init_{\mu}x$ over $G_\nu$ is one of the
    irreducible factors $g_j$ on the right hand side of (\ref{inihfact}).

\end{enumerate}

\end{prop}

Now let $g_1$ be the minimal polynomial of $\init_{\mu}x$ over
$G_\nu$. Let $\alpha_2=\deg_Xg_1$. Write
$g_1=\sum\limits_{i=0}^{\alpha_2}\bar b_iX^i$, where $\bar
b_{\alpha_2}=1$. For each $i$, $0\le i\le\alpha_2$, choose a representative
$b_i$ of $\bar b_i$ in $R_\nu$ (that is, an element of $R_\nu$ such that
$\init_\nu b_i=\bar b_i$; in particular, we take $b_{\alpha_2}=1$). Put
$Q_2=\sum\limits_{i=0}^{\alpha_2}b_ix^i$.
\begin{deft}(\cite{HOS}, page 1041) The elements $Q_1$ and $Q_2$ are called, respectively,
{\bf the first and second key polynomials} of $\mu$.
\end{deft}
Now, every element $y$ of $K[x]$ can be written uniquely as a finite sum of the
form
\begin{equation}\label{secondstandardexpansion}
y=\sum\limits_{\substack{0\le\gamma_1<\alpha_2 \\
0\le\gamma_2}}
b_{\gamma_1\gamma_2}Q_1^{\gamma_1}Q_2^{\gamma_2}
\end{equation}
where $b_{\gamma_1\gamma_2}\in K$ (this is proved by Euclidean division by the
monic polynomial $Q_2$). The expression (\ref{secondstandardexpansion})
is called \textbf{the second standard expansion of }$y$.\\

Now, take an ordinal number greater than or equal to 3 which has an immediate
predecessor; denote this ordinal by $l+1$. If $\nu(\mathbb{N})=0$, assume that $l\in\mathbb{N}_0$. Assume given a set $\bold Q_{l+1}$ of polynomials and positive integers ${\bold\alpha}_{l+1}=\{\alpha_i\}_{i\le l}$, such that $\mu(Q_{i})\in \Gamma'_1$ for $i\le l$ and all but finitely many of
the $\alpha_i$ are equal to 1. Furthermore, we assume that for each $i\le l$ the polynomial $Q_i$ has an explicit expression in terms of $\bold Q_i$, described below.

We will use the following multi-index notation:
$\bar\gamma_{l+1}=\{\gamma_i\}_{i\le l}$, where all but finitely
many $\gamma_i$ are equal to 0, $\bold
Q_{l+1}^{{\bar\gamma}_{l+1}}=\prod\limits_{i\le
l}Q_i^{\gamma_i}$. Let $\beta_i=\mu(Q_i)$.
\begin{deft}(\cite{HOS}, page 1041) An index $i<l$ is said to be $l$-{\bf essential} if there
exists a positive integer $t$ such that either $i+t=l$ or $i+t<l$ and
$\alpha_{i+t}>1$; otherwise $i$ is called $l$-{\bf inessential}.
\end{deft}
In other words, $i$ is $l$-inessential if and only if $i+\omega\le l$ and
$\alpha_{i+t}=1$ for all $t\in\Bbb N_0$.

\noi{\bf Notation.} For $i<l$, let

\begin{align*}
i+&=i+1\qquad\,\text{if }i\text{ is l-essential}\\
&=i+\omega\qquad\text{otherwise}.
\end{align*}

\begin{deft}(\cite{HOS}, page 1041)

A multiindex ${\bar\gamma}_{l+1}$ is said to be \textbf{standard with respect to}
${\bold\alpha}_{l+1}$ if

\begin{equation}\label{gammaalpha}
0\le\gamma_i<\alpha_{i+}\text{ for }i\le l,
\end{equation}
and if $i$ is $l$-inessential then the set $\{j<i+\ |\ j+=i+\text{ and
}\gamma_j\ne0\}$ has cardinality at most one. An $l$\textbf{-standard monomial in}
$\bold Q_{l+1}$ (resp. an $l${\bf-standard monomial in} $\init_{\mu}\bold Q_{l+1}$)
is a product of the form $c_{{\bar\gamma}_{l+1}}\bold Q_{l+1}^{{\bar\gamma}_{l+1}}$,
(resp. $c_{{\bar\gamma}_{l+1}}\init_{\mu}\bold Q_{l+1}^{{\bar\gamma}_{l+1}}$) where
$c_{{\bar\gamma}_{l+1}}\in K$ (resp. $c_{{\bar\gamma}_{l+1}}\in G_\nu$) and the
multiindex ${\bar\gamma}_{l+1}$ is standard with respect to ${\bold\alpha}_{l+1}$.
\end{deft}

\begin{rek}(\cite{HOS}, page 1042) In the case when $i$ is $l$-essential, the
condition (\ref{gammaalpha}) amounts to saying that $0\le\gamma_i<\alpha_{i+1}$.
\end{rek}

\begin{deft}\label{standardexpansionnotinvolving}(\cite{HOS}, page 1042) An $l${\bf-standard expansion not involving }$Q_l$ is a finite
sum $S$ of $l$-standard monomials, not involving $Q_l$, having the following property.
Write $S=\sum\limits_\beta S_\beta$, where $\beta$ ranges over a certain finite
subset of $\Gamma'_+$ and
\begin{equation}\label{Sbeta}
S_\beta=\sum\limits_jd_{\beta j}
\end{equation}
is a sum of standard monomials $d_{\beta j}$ of value $\beta$. We require that
\begin{equation}\label{inimudbetaj}
\sum\limits_j\init_{\mu}d_{\beta j}\ne0
\end{equation}
for each $\beta$ appearing in (\ref{Sbeta}).
\end{deft}
\begin{prop}(\cite{HOS}, page 1042) Let $i$ be an ordinal and $t$ a positive integer. Assume that $i+t+1\le l$, so that the key
polynomials $\bold Q_{i+t+1}$ are defined, and that $\alpha_i=\dots=\alpha_{i+t}=1$. Then any
$(i+t)$-standard expansion does not involve any $Q_q$ with $i\le q<i+t$. In
particular, an $i$-standard expansion not involving $Q_i$ is the same thing as an
$(i+t)$-standard expansion, not involving $Q_{i+t}$.
\end{prop}

We will frequently use this fact in the sequel without mentioning it explicitly.
\begin{deft}(\cite{HOS}, page 1042)
For an element $g\in K[x]$, an expression of the form
$g=\sum\limits_{j=0}^sc_jQ_l^j$, where each $c_j$ is an $l$-standard expansion not
involving $Q_l$, will be called an $l$-{\bf standard expansion of }$g$.
\end{deft}

\begin{deft}(\cite{HOS}, page 1042) Let $\sum\limits_{{\bar\gamma}}\bar
c_{{\bar\gamma}}\init_{\mu}\bold Q_{l+1}^{{\bar\gamma}}$ be an $l$-standard
expansion, where $\bar c_{{\bar\gamma}}\in G_\nu$. A {\bf lifting} of
$\sum\limits_{{\bar\gamma}}\bar c_{{\bar\gamma}}\init_{\mu}\bold
Q_{l+1}^{{\bar\gamma}}$ to $K[x]$ is an $l$-standard expansion
$\sum\limits_{{\bar\gamma}}c_{{\bar\gamma}}\bold Q_{l+1}^{{\bar\gamma}}$, where
$c_{{\bar\gamma}}$ is a representative of $\bar c_{{\bar\gamma}}$ in $K$.
\end{deft}

\begin{deft}(\cite{HOS}, page 1042) Assume that $char\ k_\nu=p>0$. An $l$-standard expansion
  $\sum\limits_jc_jQ_l^j$, where each $c_j$ is an $l$-standard expansion not
  involving $Q_l$, is said to be {\bf weakly affine} if $c_j=0$ whenever
  $j>0$ and $j$ is not of the form $p^e$ for some $e\in\Bbb N_0$.
\end{deft}

Assume, inductively, that for each ordinal $i\le l$, every element $h$ of $K[x]$
admits an $i$-standard expansion. Furthermore, assume that for each $i\le l$, the
$i$-th polynomial $Q_i$ admits an $i_0$-standard expansion, with $i=i_0+$,
having the following additional properties:\\

If $i$ has an immediate predecessor $i-1$ in $\Lambda$ (such is always the
case if $char\ k_\nu=0$), the $(i-1)$-st standard expansion of $Q_i$ has the form
\begin{equation}\label{standardform}
Q_i=Q_{i-1}^{\alpha_i}+\sum\limits_{j=0}^{\alpha_i-1}
\left(\sum\limits_{{\bar\gamma}_{i-1}}
c_{ji{\bar\gamma}_{i-1}}\bold
Q_{i-1}^{{\bar\gamma}_{i-1}}\right)Q_{i-1}^j,
\end{equation}
where:

\begin{enumerate}
\item each $c_{ji{\bar\gamma}_{i-1}}\bold
Q_{i-1}^{{\bar\gamma}_{i-1}}$ is an $(i-1)$-standard monomial, not involving
$Q_{i-1}$
\item the quantity
$\nu\left(c_{ji{\bar\gamma}_{i-1}}\right)+j\beta_{i-1}+
\sum\limits_{q<i-1}\gamma_q\beta_q$ is constant for all the monomials
$$
\left(c_{ji{\bar\gamma}_{i-1}}\bold Q_{i-1}^{{\bar\gamma}_{i-1}}\right)Q_{i-1}^j
$$
appearing on the right hand side of (\ref{standardform})
\item the equation
\begin{equation}\label{standardformini}
\init_{\mu}Q_{i-1}^{\alpha_i}+\sum\limits_{j=0}^{\alpha_i-1}
\left(\sum\limits_{{\bar\gamma}_{i-1}}\init_\nu
c_{ji{\bar\gamma}_{i-1}}\init_{\mu}\bold
Q_{i-1}^{{\bar\gamma}_{i-1}}\right)
\init_{\mu}Q_{i-1}^j=0
\end{equation}
is the minimal algebraic relation satisfied by $\init_{\mu}Q_{i-1}$ over
$G_\nu[\init_{\mu}\bold Q_{i-1}]$.
\end{enumerate}

Finally, if $char\ k_\nu=p>0$ and $i$ does not have an immediate predecessor in
$\Lambda$ then there exist an $i$-inessential index $i_0$ and a strictly
positive integer $e_i$ such that $i=i_0+$ and
\begin{equation}\label{standardform2}
Q_i=c_{0i_0}+\sum\limits_{j=0}^{e_i}c_{p^ji_0}Q_{i_0}^{p^j}
\end{equation}
is a weakly affine monic $i_0$-standard expansion of degree $\alpha_i=p^{e_i}$ in
$Q_{i_0}$, where each $c_{qi_0}$ is an $i_0$-standard expansion not involving
$Q_{i_0}$. Moreover, there exists a positive element $\bar\beta_i\in\Gamma'$ such that

\begin{align*}
\bar\beta_i&>\beta_q\qquad\text{ for all }q<i,\\
\beta_i&\ge p^{e_i}\bar\beta_i\quad\text{ and}\\
p^j\bar\beta_i+\nu(c_{p^ji_0})&=p^{e_i}\bar\beta_i\quad\text{ for }0\le j\le
e_i.
\end{align*}

\begin{deft}
 The set $\bold Q_{l+1}$ is called an $l$-th set of \textbf{HOS Key Polynomial}. By \textbf{a set of HOS key polynomials} we will mean a set $\bold Q$ of polynomials for which there exists an ordinal $l$ such that $\bold Q$ is an $l$-th set of key polynomials. We will loosely refer to elements of this set as \textbf{HOS key polynomials}.
\end{deft}

If $i\in\Bbb N_0$, one can prove by induction that the $i$-standard expansion is
unique. If $char\ k_\nu>0$ and $h=\sum\limits_{j=0}^{s_i}d_{ji}Q_i^j$ is an $i$-standard expansion of $h$ (where
$h\in K[x]$), then the elements $d_{ji}\in K[x]$ are uniquely determined
by $h$ (strictly speaking, this does not mean that the $i$-standard expansion is
unique: for example, if $i$ is a limit ordinal, $d_{ji}$ admits an $i_0$-standard
expansion for each $i_0<i$ such that $i=i_0+$, but there may be countably many
choices of $i_0$ for which such an $i_0$-standard expansion is an $i_0$-standard
expansion, not involving $Q_{i_0}$ in the sense of Definition \ref{standardexpansionnotinvolving}).\\

\begin{deft}\label{nui}
For each ordinal $i\le l$ we define a valuation $\nu_i$ of $L$
as follows. Given an $i$-standard expansion $h=\sum\limits_{j=0}^{s_i}d_{ji}Q_i^j$, put
\begin{align}\label{eq:nui}
\nu_i(h)=\min\limits_{0\le j\le s_i}\{j\beta_i+\mu(d_{ji})\}.
\end{align}
The valuation $\nu_i$ will be called the $i$-{\bf truncation} of $\nu$.
\end{deft}
Note that even though in the case when $char\ k_\nu>0$ the standard
expansions of the elements $d_{ji}$ are not, in general, unique, the elements
$d_{ji}\in K[x]$ themselves are unique by Euclidean division, so $\nu_i$ is well
defined. That $\nu_i$ is, in fact, a valuation, rather than a
pseudo-valuation, follows from the definition of standard expansion,
particularly, from (\ref{inimudbetaj}). We always have
$$
\nu_i(h)\le\mu(h).
$$

The paper \cite{HOS} constructs, starting with a set $\bold Q_{l+1}$ of HOS key polynomials, a polynomial $Q_{l+1}$ such that $\bold Q_{l+2}$ forms and $(l+1)$-st set of key polynomials (this means, in other words, that $Q_{l+1}$ which has the form (\ref{standardform}) and satisfies properties 1--3). If $\alpha_{l+1}=1$ it may happen that the construction of \cite{HOS} produces a $(l+\omega+1)$-st set of HOS key polynomials, that is, an infinte
 sequence of polynomials $\left\{Q_{l+t}\right\}_{t\in\mathbb{N}}$ and a polynomial $Q_{l+\omega}$ (referred to as the limit HOS key polynomial) which has the form (\ref{standardform2}) and satisfies the properties listed right after equation (\ref{standardform2}).\\

 We will finish our construction here, for more details of the construction, we refer the reader
 to the paper \cite{HOS}. We also note that the paper \cite{HOS} section 8 page 1068 proves that we can always construct a \textbf{complete} family of HOS key polynomials.  \\

 We end this section giving more definitions and results of \cite{HOS}
 that we will use in the rest of our paper.

\medskip

\begin{prop}\label{prop21}(\cite{HOS}, page 1044)
\begin{enumerate}
	\item The polynomial $Q_{i}$ is monic in $x$; we have
	\[	\deg_{x}Q_{i}=\prod\limits_{j\leq i}\alpha_{j}.
\]
	\item Let $z$ be an $i$-standard expansion, not involving $Q_{i}$. Then
	\[
		\deg_{x}z<\deg_{x}Q_{i}.
\]
\end{enumerate}
\end{prop}
\medskip

We recall a result from the statement of Corollary 25 in \cite{HOS}:
\begin{prop}\label{cor25}(\cite{HOS}, page 1045, Corollary 25)
We have
\begin{align*}
\beta_{i}>\alpha_{i}\beta_{i-1}&\ \ \ \ \ if\ (i-1)\ exists&\\
\beta_{i}>p^{e_{i}}\overline{\beta_{i}}&\ \ \ \ \ otherwise.&
\end{align*}
\end{prop}

\begin{prop}\label{prop36}(\cite{HOS}, page 1051) Consider an ordinal $l\in\Lambda$.
Let $y$ be a polynomial in $K[x]$ of degree strictly less than $\deg_{x}(Q_{l+1})=\prod\limits_{i=1}^{l+1}\alpha_{i}$. Then $\mu(y)=\nu_{l}(y)$.
\end{prop}

\begin{deft}
Let $h=\sum\limits_{j=0}^sd_{ji}Q_{i}^{j}$ be an $i$-standard expansion, let $\overline{Q}_{i}$ be a variable, and let $\beta_{i}=\mu(Q_{i})$.
We define $$ S_{i}(h, \beta_{i}): = \left\{ j\in\left\{0,...,s\right\}\ |\ j\beta_{i}+\mu(d_{ji})=\nu_{i}(h) \right\} $$
$$
in_{i}(h):= \sum_{j\in S_{i}(h, \beta_{i})} in_{\mu}d_{ji} \overline{Q}_{i}^j
$$
We define $\delta_{i}(h):=\deg_{\overline{Q}_{i}}in_{i}(h)$.
\end{deft}

We recall a result from the statement of Proposition 37 in \cite{HOS}:
\begin{prop}\label{prop37}(\cite{HOS}, page 1044)
We have $\alpha_{i+1}\delta_{i+1}(h)\leq \delta_{i}(h)$.
\end{prop}

\section{The main results: comparison of Vaqui\'e and HOS key polynomials.}\label{Comparison}

Let the notation be as in the previous sections, with $rk\ \nu=1$.
\begin{prop}\label{HOSimpliesVaquie}
We assume that the family of HOS key polynomials
$\textbf{Q}_{i}=\left\{Q_{i}\right\}_{i\in\Lambda}$ is already
defined (\cite{HOS}, part 3). Let $i$ be an ordinal and let $i_0=i-1$ if $i$ admits
an immediate predecessor and $i_0$ as in (\ref{standardform2}) otherwise.

Then $Q_i$ is a Vaqui\'e key polynomial for the valuation $\nu_{i_0}$.

\end{prop}

\begin{proof}:
\begin{enumerate}
\item $Q_{i}$ is $\nu_{i_{0}}$-minimal:\\
 Suppose that $Q_i$ $\nu_{i_0}$-divides $f$, then there exists $h\in K[x]$ such that
	\[ \nu_{i_0}(f-hQ_i)>\nu_{i_0}(f)=\nu_{i_0}(hQ_i).
\]

On the other hand we have: $\mu(hQ_i)=\mu(h)+\beta_i>\nu_{i_0}(h)+\alpha_{i}\beta_{i_0}=\nu_{i_0}(hQ_i)=\nu_{i_0}(f)$ where the strict inequality holds by Proposition \ref{cor25} .\\
Then we have the inequality
\[\mu(f)\geq\inf\{\mu(f-hQ_i), \mu(hQ_i)\}
			\geq\inf\{\nu_{i_0}(f-hQ_i), \mu(hQ_i)\}
			>\nu_{i_0}(f).\]
Then $\deg_{x}f\geq\deg_{x}Q_{i}$ because otherwise by Proposition \ref{prop36}
we would have
$$
\mu(f)=\nu_{i_0}(f).
$$	
\item $Q_{i}$ is $\nu_{i_{0}}$-irreducible:\\
Suppose that $Q_i$ $\nu_{i_0}$-divides $f.g$, as above, we find
\[
	\mu(f.g)>\nu_{i_0}(f.g)
	\]
	Now, either $\mu(f)>\nu_{i_0}(f)$ and
        $\deg(f)\geq\deg(Q_{i})$, or $\mu(g)>\nu_{i_0}(g)$ and
        $\deg(g)\geq\deg(Q_{i})$.\\
	Suppose that $\mu(f)>\nu_{i_0}(f)$ and let $f=qQ_{i}+r$ be
        the Euclidean division of $f$ by $Q_i$ with $q\neq0$ because
        $\deg(f)\geq\deg(Q_{i})$, and $\mu(r)=\nu_{i_0}(r)$
        because $\deg(r)<\deg(Q_{i})$.\\
	Hence
\[ \nu_{i_0}(r)=\mu(r)\geq\inf(\mu(f),
\mu(qQ_i))>\inf(\nu_{i_0}(f), \nu_{i_0}(qQ_i))
	\]
	\[ \Rightarrow \nu_{i_0}(r)>\nu_{i_0}(f)=\nu_{i_0}(qQ_i) \]
	and $Q_{i}$ $\nu_{i_0}$-divides $f$.
	\item $Q_{i}$ is monic by definition.
\end{enumerate}

\vspace{-1.0cm}\[\qedhere\]				
\end{proof}

\begin{rek}
	Let $i$ and $i_0$ be as in Proposition \ref{HOSimpliesVaquie}.
	As HOS key polynomials are also Vaqui\'e key polynomials for $\nu_{i_0}$, and as $\mu(d_{ji})$ on the right side of (\ref{eq:nui}) in Definition \ref{nui} is equal to $\nu_{i_{0}}(d_{ji})$ by Proposition \ref{prop36}, the $i$-truncation $\nu_{i}$ is also an augmented valuation defined by the valuation $\nu_{i_{0}}$ and the key polynomial $Q_{i}$.
\end{rek}

\begin{cor}\label{eliminatingsamedegree}
Assume that there exists an ordinal $i$ and a strictly positive integer $t$ such that $deg_{x}Q_{i}=deg_{x}Q_{i+t}$.
Then the polynomial $Q_{i+t}$ is also a Vaqui\'e key polynomial for the valuation $\nu_{i_0}$.
\end{cor}
\begin{proof} The polynomial $Q_{i+t}$ can be written as $Q_{i+t}=Q_{i}+z_{t}$ where $z_{t}$ is an $i$-standard expansion not involving $Q_{i}$ with $\nu_{i}(Q_{i})=\nu_{i}(z_{t})$. We have $\nu_{i_0}(z_{t})=\nu_{i}(z_{t})$ by Proposition \ref{prop36}. Now, $\nu_{i}(Q_{i})>\nu_{i_0}(Q_{i})$, so $\nu_{i_0}(z_{t})>\nu_{i_0}(Q_{i})=\nu_{i_0}(Q_{i+t})$. Hence $Q_{i}$ and $Q_{i+t}$ are $\nu_{i_0}$-equivalent.
\end{proof}

Let $\{Q_{i}\}_{i\in\Lambda}$ be a family of HOS key polynomials constructed in \cite{HOS}, and $\{\nu_{i}\}_{i\in\Lambda}$ the correspending family of valuations.

Take an ordinal $i+1\in\Lambda$ which admits an immediate predecessor $i$, and such that $\deg_{x}Q_{i}=\deg_{x}Q_{i+1}$. Let $\Delta_{i}$ be a totally ordered set such that there exists a bijection $\Phi$ between $[\beta_{i};\beta_{i+1}]\subset\nu_{i+1}(K[x])$ and $\Delta_{i}$.

\begin{lem}\label{deltai}
For all $\beta\in]\beta_{i};\beta_{i+1}[$, there exists a polynomial $Q_{\delta}\in K[x]$ with $\delta=\Phi(\beta)$ in $\Delta_{i}$  which satisfies :
$$
Q_{\delta}=Q_{i}+z_{\delta}
$$
with $z_{\delta}$ an i-standard expansion not involving $Q_{i}$,
$$
\nu_{i}(Q_{\delta})=\nu_{i}(Q_{i})=\nu_{i}(z_{\delta})
$$
and
$$
\nu_{i+1}(Q_{\delta})=\beta_{\delta}>\beta_{i}.
$$
\end{lem}
\begin{proof}
Take any $\beta\in]\beta_{i};\beta_{i+1}[$, put $\delta=\Phi(\beta)\in \Delta_{i}$.\\
As $\beta\in\nu_{i+1}(K[x])$ then there exist $h_{\beta}\in K[x]$ such that $\nu_{i+1}(h_{\beta})=\beta$. As $\nu_{i+1}(h_{\beta})=\beta<\beta_{i+1}$ then $\mu(h_{\beta})=\nu_{i+1}(h_{\beta})$ and by Lemma \ref{lemme1.4}
there exists an $h_{\delta}\in K[x]$, such that
$$
\deg_{x}(h_{\delta})<\deg_{x}(Q_{i+1})=\deg_{x}(Q_{i})
$$
and $\mu(h_{\delta})=\mu(h_{\beta})=\beta$.\\
 Put $Q_{\delta}=Q_{i+1}+h_{\delta}$.
We notice that $\nu_{i+1}(Q_{\delta})=\nu_{i+1}(h_{\delta})=\mu(h_{\delta})<\beta_{i+1}=\nu_{i+1}(Q_{i+1})$.\\
As $Q_{i+1}=Q_{i}+z_{i}$ with $deg_{x}(z_{i})<deg_{x}(Q_{i})$, and $\nu_{i}(Q_{i})=\mu(z_{i})$,\\
then $Q_{\delta}=Q_{i}+(z_{i}+h_{\delta})$, therefore $\nu_{i}(Q_{\delta})=\nu_{i}(Q_{i})=\beta_{i}$ because $\nu_{i}(h_{\delta})=\beta>\beta_{i}$.\\
Hence, $Q_{\delta}=Q_{i}+z_{\delta}$, with $z_{\delta}=h_{\delta}+z_{i}$, with $deg_{x}(z_{\delta})<deg_{x}(Q_{i})$ and $$
\nu_{i}(Q_{\delta})=\nu_{i}(Q_{i})<\nu_{i+1}(Q_{\delta})=\mu(Q_{\delta}).
$$
Put $\beta_{\delta}=\mu(Q_{\delta})=\beta$, then we have
 $$
 Q_{\delta}=Q_{i}+z_{\delta}
 $$
 with $z_{\delta}$ an i-standard expansion not involving $Q_{i}$,
 $$
 \nu_{i}(Q_{\delta})=\nu_{i}(Q_{i})=\nu_{i}(z_{\delta})
 $$
 and $\beta_{\delta}>\beta_{i}$.

\end{proof}

Fix $\delta\in\Delta_{i}$ and take the polynomial $Q_{\delta}$ defined above.
Put $\nu_{\delta}=[\nu_{i};\nu_{\delta}(Q_{\delta})=\beta_{\delta}]$.

By Lemma \ref{deltai} and proposition \ref{HOSimpliesVaquie}, the polynomials $\{Q_{i}\}_{i\in\Delta_{i}}$ are Vaqui\'e key polynomials for the valuation $\nu_{i}$.

\begin{prop}\label{familyqdelta}

The family $(\nu_\delta)_{\delta\in\Delta_{i}}$ associated to the key polynomials $(Q_\delta)_{\delta\in\Delta_{i}}$ is an augmented iterated family of valuations.

\end{prop}
\begin{proof}:

Take $\delta_{1}<\delta_{2}\in\Delta_{i}$, we have $Q_{\delta_{1}}=Q_{i}+z_{\delta_{1}}$, and $Q_{\delta_{2}}=Q_{i}+z_{\delta_{2}}$. Therefore $Q_{\delta_{2}}=Q_{\delta_{1}}+(z_{\delta_{2}}-z_{\delta_{1}})$, from this relation we can see that the polynomial $Q_{\delta_{2}}$ is a key polynomial for the valuation $\nu_{\delta_{1}}$ and that the valuation $\nu_{\delta_{2}}$ is the augmented valuation constructed by the valuation $\nu_{\delta_{1}}$ and the key polynomial $Q_{\delta_{2}}$.

Now as we have, $deg_{x}(Q_{\delta})=\deg_{x}(Q_{i})$, $\forall\delta\in\Delta_{i}$,
we still have to prove that $Q_{\delta_{1}}$ and $Q_{\delta_{2}}$ are not $\nu_{\delta_{1}}$-equivalent.
We have, $z_{\delta_{1}}=z_{i}+h_{\delta_{1}}$ and $z_{\delta_{2}}=z_{i}+h_{\delta_{2}}$ with $\mu(h_{\delta_{1}})=\Phi^{-1}(\delta_{1})$ and $\mu(h_{\delta_{2}})=\Phi^{-1}(\delta_{2})$, and $\Phi^{-1}(\delta_{1})<\Phi^{-1}(\delta_{2})$.
Hence $\nu_{\delta_{1}}(Q_{\delta_{2}}-Q_{\delta_{1}})=\nu_{\delta_{1}}(z_{\delta_{2}}-z_{\delta_{1}})=
\nu_{\delta_{1}}(h_{\delta_{2}}-h_{\delta_{1}})=\mu(h_{\delta_{2}}-h_{\delta_{1}})=\Phi^{-1}(\delta_{1})=
\nu_{\delta_{1}}(Q_{\delta_{1}})=\nu_{\delta_{1}}(Q_{\delta_{2}})$.

This completes the proof.

\vspace{-1.0cm}\[\qedhere\]
\end{proof}

The result of Proposition \ref{familyqdelta} is  is closely related to the result of Proposition \ref{exhaustion} (due to Vaqui\'e). Although the latter is sufficient for the purposes of this paper, we have kept Proposition \ref{familyqdelta} because we feel that it clarifies the nature of $Exh(F)$ and the relation between Vaqui\'e and HOS key polynomials.\\

In fact, by Proposition \ref{exhaustion} we can extend any continued family $F$ of augmented iterated valuations to an exhaustive family $Exh(F)$.

We note, using Proposition \ref{HOSimpliesVaquie}, that the polynomial $Q_{l+\omega}$ defined in the seventh part of \cite{HOS} is a Vaqui\'e limit key polynomial for the family $\{Q_{l+t}\}_{t\in\N_{0}}$.

\begin{thm}\label{admitted} The family $F=\{\nu_{i}\}_{i\in\Lambda}$ constructed in \cite{HOS} can be extended to an admitted family $Exh(F)$ for the valuation $\mu$.
\end{thm}

\begin{proof}
We will proceed by the order of construction of the $Q_{i}$.\\
\\
We take $Q_{1^{(1)}}=x$ and $\nu_{1^{(1)}}$. If $\nu_{1^{(1)}}(h)$=$\mu(h)$ $\forall h\in K[x] $ we have finished, we take $Exh(F)=\left\{\nu_{1^{(1)}}\right\}$.

If not, consider the polynomial $Q_{2}$. If there exists an integer $t_0$ such that $deg_{x}(Q_{2+t})=deg_{x}Q_{2}$ for all $t\leq t_{0}$, $Q_{2+t_{0}+1}$ is defined and $deg_{x}(Q_{2+t_{0}+1})>deg_{x}Q_{2}$, then by Corollary \ref{eliminatingsamedegree} the polynomial $Q_{2+t_0}$ is a Vaqui\'e key polynomial for the valuation $\nu_{1^{(1)}}$. We put $Q_{2^{(1)}}=Q_{2+t_0}$ and $\nu_{2^{(1)}}=[\nu_{1^{(1)}};
\nu_{2^{(1)}}(Q_{2^{(1)}})=\mu(Q_{2^{(1)}})]$. We use the same procedure to construct the valuations $\nu_{3^{(1)}},\ \nu_{4^{(1)}}$,...\\

If we have $\alpha_{i}>1$ for infinitely many values of $i$, we set $Exh(F)=\left\{\nu_{i}^{(1)}\right\}_{i\in I^{(1)}}$, with $I=I^{(1)}=\left\{1, ..., n,...\right\}$. Take any element $h\in K[x]$. From Proposition \ref{prop37} we have
\begin{equation}
\delta_{i+1}(h)<\delta_{i}(h)\quad\text{ for }i,i+1\in I^{(1)}.\label{eq:strictineq}
\end{equation}
As the set $I^{(1)}$ is infinite and the strict inequality (\ref{eq:strictineq}) cannot occur infinitely many times, we have $\delta_{i}(h)=0$ for some $i$. Then $in_{i}(h)$ does not involve $\overline{Q}_{i}$, hence $\nu_{i}(h)=\mu(h)$ and we have finished.\\
If not, i.e. if there exists a certain $l$ such that $\alpha_{l}=\alpha_{l+1}=\alpha_{l+2}=....=1$, we set
$$
I^{(1)}=B^{(1)}\bigcup Exh\left(A^{(1)}\right)
$$
with $B^{(1)}=\left\{1^{(1)},...,l^{(1)}\right\}$ and $A^{(1)}=\left\{l^{(1)}+1,l^{(1)}+2,.... \right\}$ where $l^{(1)}$ is the minimal $l$ satisfying $\alpha_{l+t}=1$ for all $t\in\mathbb{N}_0$. \\

If $\forall h$ in $K[x]$, there exists $i\in A^{(1)}$ such as $\nu_{i}(h)=\mu(h)$ we have finished.\\
If not, we know the existence of a limit key polynomial $Q_{l+\omega}$ and a valuation limit $\nu_{l+\omega}$, which satisfies $\nu_{l+\omega}(f)\leq \mu(f)$ for all $f\in K[x]$. We denote: $\nu_{l+\omega}=\nu_{1}^{(2)}$ and we repeat the procedure.\\
In this way, we construct recursively an admissible family of augmented iterated valuations which is admitted for the valuation $\mu$.
\end{proof}

Conversely, given an admissible family of valuations $F$ of $K[x]$ which is admitted for the valuation $\mu$, we want to see how to obtain from the family of Vaqui\'e key polynomials associated to $F$, a family of HOS key polynomials.

We will first prove an analogue of Lemma \ref{lemme1.4} when $Q_{i}$ is a limit key polynomial.

\begin{lem}\label{lemme_limit}
Let $C=\{\mu_{\alpha}\}_{\alpha\in A}$ be a continued family of augmented iterated valuations, and $\{\phi_{\alpha}\}_{\alpha\in A}$ the set of the key polynomials associated to $C$.\\
 Let $\mu$ be the valuation defined by the family $C$, a limit key polynomial $\phi$ and a value $\gamma=\mu(\phi)$.
Then for all $f$ in $K[x]$ for which there exists $\alpha_{0}\in A$ such that for all $\alpha\geq\alpha_{0}$ $\mu_{\alpha}(f)=\mu(f)$,
 we have:
\begin{enumerate}
\item there exists $h$ in $K[x]$ with $\deg h<\deg\phi$ such that $in_{\mu}f=in_{\mu}h$.
\item there exists $g$ in $K[x]$ with $\deg g<\deg\phi$ such that $in_{\mu}fg=in_{\mu}1$.
\end{enumerate}

\end{lem}

\begin{proof}
\begin{enumerate}
\item Let $f=q\phi+r$ be the Euclidean division of $f$ by $\phi$.\\
As $\deg_{x}r<\deg_{x}\phi$, there exists $\alpha_{1}\in A$ such that for all $\alpha\geq\alpha_{1}$ in $A$, we have $\mu_{\alpha}(r)=\mu(r)$.\\
Take $\alpha_{2}=max\left\{\alpha_{0}, \alpha_{1}\right\}$
 then for all $\alpha\geq\alpha_{2}$, we have $\mu_{\alpha}(q\phi)\geq\mu_{\alpha}(f)$. \\
Indeed, suppose that there exists $\beta\in A$, $\beta\geq\alpha_{2}$ and $\mu_{\beta}(q\phi)<\mu_{\beta}(f)$. Then $\mu_{\beta}(r)=\mu_{\beta}(q\phi)$
and for all $\alpha\geq\beta\geq\alpha_{2}$ we have $\mu_{\alpha}(f)\geq\mu_{\beta}(f)>\mu_{\beta}(r)=\mu_{\alpha}(r)$.
 Hence for all $\alpha\geq\beta$ we have $in_{\mu_{\alpha}}r=in_{\mu_{\alpha}}q\phi$ and $\phi$ $A$-divides $r$, which contradicts the fact that $\deg_{x}r<\deg_{x}\phi$ because $\phi$ is $\mu_{A}$-minimal.\\
Hence, for all $\alpha\geq\alpha_{2}$ we have $\mu(q\phi)>\mu_{\alpha}(q\phi)\geq\mu_{\alpha}(f)=\mu(f)$, therefore $f$ is $\mu$-equivalent to $r$.

\item As the limit key polynomial $\phi$ is an irreducible polynomial of $K[x]$ and $\phi$ does not divide $f$, therefore there exist two polynomials $g$ and $h$ of $K[x]$, with $deg_{x}g<deg_{x}\phi$, such that $fg+h\phi=1$, hence $in_{\mu}fg=in_{\mu}1$.
\end{enumerate}

\end{proof}

Let $F=\{\mu_{i}\}_{i\in I}$ be an admissible family of valuations of $K[x]$ which is admitted for $\mu$, and let $\left\{Q_{i}\right\}_{i\in I}$ be the family of key polynomials in the sense of Vaqui\'e associated to $F$. For all $i\in I$ put $\beta_{i}=\mu_{i}(Q_{i})$.\\

Write $F=\bigcup\limits_{t}S^{(t)}=\bigcup\limits_{t}\left\{\mu_{i}^{(t)}\right\}_{i\in I^{(t)}}$, with $1\leq t<N$ where $N \in\N\bigcup \left\{+\infty\right\}$, and
$$
I^{(t)}=\left\{ 1^{(t)},...,n^{(t)}\right\}\bigcup A^{(t)}.
$$
\begin{thm}\label{VaquietoHOS} There exist well ordered sets $I'$ and
$J$, $I'\subset I$, $I'\subset J$, and a polynomial $Q'_i$ for each
$i\in J$, having the following properties:
\begin{enumerate}
\item $I'$ is cofinal in both $I$ and $J$.
\item $1^{(t)}\in I'$ for all $t$, $1\le t<N$.
\item the set $(Q'_{i})_{i\in J}$ is a $J$-set of HOS key polynomials.
\item for any two consecutive elements  $i_0,i_1\in I'$, there is at most one $j\in J$ such that
\begin{equation}
i_0<j<i_1.\label{eq:jbetween}
\end{equation}
If there exists $j\in J$ satisfying (\ref{eq:jbetween}) then $\deg\
Q'_j=\deg\ Q_{i_1}$.
\item For every $i\in I'$ there exists $i_0\in I'$ with $i=i_0+$ such that the key polynomial $Q'_{i}$ satisfies
$\mu_{i_0}(Q_{i}-Q'_{i})>\mu(Q_{i})=\mu(Q'_{i})$.
\item If the family $F$ converges to $\mu$ then the set
$(Q'_{i})_{i\in J}$ of HOS key polynomials is complete.
\end{enumerate}
\end{thm}
\begin{proof} We start the proof of Theorem \ref{VaquietoHOS} with an auxiliary Proposition which gives explicit formulae expressing each Vaqui\'e key polynomial $Q_i$ appearing in $F$ in terms of polynomials $Q_{i'}$ with $i'<i$.

\begin{prop}\label{VaquieimpliesHOS}
For every $i$ in $I$ there exists an $i_0\in I$, $i_{0}<i$ such that the polynomial $Q_i$ can be written as:
\begin{equation}
Q_{i}=Q_{i_0}^{\alpha_{i}} + \sum\limits_{j=0}^{\alpha_{i}-1}\left(\sum\limits_{\gamma_{i_0}}c_{ji\gamma_{i_0}}\textbf{Q}_{i_0}^{\gamma_{i_0}}
\right)Q_{i_0}^j\label{eq:standardform}
\end{equation}
where:
\begin{enumerate}
\item Each $c_{ji\gamma_{i_0}}\textbf{Q}_{i_0}^{\gamma_{i_0}}$ is an $i_{0}$-standard monomial not involving $Q_{i_0}$.
	\item We have $j\beta_{i_0}+ \mu\left(c_{ji\gamma_{i_0}}\textbf{Q}_{i_0}^{\gamma_{i_0}}\right)\ge\alpha_i\beta_{i_0}$ for all the monomials $\left(c_{ji\gamma_{i_0}}\textbf{Q}_{i_0}^{\gamma_{i_0}}\right)Q_{i_0}^j$ appearing in (\ref{eq:standardform}).
	\item We have $\beta_i>\alpha_i\beta_{i_0}$.
	\item $Q_{i}$ is a polynomial of minimal degree among those satisfying $\nu_{i_{0}}(Q_{i})<\mu(Q_{i})$.
\end{enumerate}
\end{prop}

\begin{proof}
Write $F=\bigcup\limits_{t}S^{(t)}=\bigcup\limits_{t}\left\{\mu_{i}^{(t)}\right\}_{i\in I^{(t)}}$, with $1\leq t<N$ where $N \in\N\bigcup \left\{+\infty\right\}$, and
$$
I^{(t)}=\left\{ 1^{(t)},...,n^{(t)}\right\}\bigcup A^{(t)}.
$$
We know that the first valuation, $\nu_{1}^{1}$ is constructed in the same way by M.Vaqui\'e and by HOS, with the key polynomial $Q_{1}^{1}=x$ or $Q_{1}^{1}=x-a$ with $a\in K$.

Now we have three cases:

\textbf{Case 1} $i=i^{(t)}\in I^{(t)}$ with $i\in \left\{ 2^{(t)},...,n^{(t)}\right\}$.\\

Put $i_{0}=(i-1)^{(t)}$, and let $Q_{i}=f_{m}Q_{i_0}^m + f_{m-1}Q_{i_0}^{m-1}+ ... + f_{0}$.\\
In the case when $Q_{i_0}$ is a key polynomial (in the sense of Vaqui\'e), Vaqui\'e proves in \cite{V} (Th\'eor\`eme 1.11, page 9) that $f_{m}=1$ and that $\mu_{i_{0}}(Q_{i})=m\beta_{i_{0}}$ where $\beta_{i_{0}}=\mu_{i_{0}}(Q_{i_0})=\mu(Q_{i_0})$.

Now if $i_0$ is a limit ordinal and $Q_{i_0}$ is a limit key polynomial (in the sense of Vaqui\'e), this means that we are in the case $t\neq1$ and $(i-2)^{(t)}$ does not exist. We will first prove the Proposition in this case.

We will denote $A=A^{(t-1)}$.\\

As $\deg_{x}f_m<\deg_{x}Q_{i_0}$, there exists $\alpha_{0}\in A$ such that for all $\alpha \in A$, $\alpha\geq\alpha_{0}$, we have $\mu_{i_0}(f_m)=\mu_{\alpha}(f_m)$. By lemma \ref{lemme_limit}
 there exists $g$ in $K[x]$ with $deg_{x}g<deg_{x}Q_{i_0}$ such that $in_{\mu_{i_0}}f_{m}g=in_{\mu_{i_0}}1$.\\
As $deg_{x}g<deg_{x}Q_{i_0}$, there exists $\alpha_{1}\in A$ such that for all $\alpha \in A$, $\alpha\geq\alpha_{1}$, we have $\mu_{i_0}(g)=\mu_{\alpha}(g)$.\\
As for all $j$, $0\leq j\leq m-1$, $deg_{x}f_{j}<deg_{x}Q_{i_0}$, then for all $j$, $0\leq j\leq m-1$, there exists $\alpha_{1,j}\in A$ such that for all $\alpha \in A$, $\alpha\geq\alpha_{1,j}$, we have $\mu_{i_0}(f_{j})=\mu_{\alpha}(f_{j})$.\\
Take $\alpha_{2}=max\left\{\alpha_{1}, \alpha_{1,0},...,\alpha_{1,j},...,\alpha_{1,m-1}\right\}$,
 then for all $\alpha\geq\alpha_{2}$, and for all $j$, $0\leq j\leq m-1$, $\mu_{i_0}(f_{j}g)=\mu_{\alpha}(f_{j}g)$,
therefore by Lemma \ref{lemme_limit}, for all $j$, $0\leq j\leq m-1$, there exists $h_{j}$ in $K[x]$
with $\deg_{x}h_j<\deg_{x}Q_{i_0}$ such that $in_{\mu_{i_0}}f_{j}g=in_{\mu_{i_0}}h_{j}$.\\
Put $\varphi=Q_{i_0}^m + h_{m-1}Q_{i_0}^{m-1}+ ... + h_{0}Q_{i_0}$, we have $in_{\mu_{i_0}}\varphi=in_{\mu_{i_0}}gQ_{i}$ so that $Q_{i}$ $\mu_{i_{0}}$-divides $\varphi$.
Therefore $\deg_{x}\varphi\geq Q_{i}$, hence $\deg_{x}f_{m}=0$, and as $Q_{i}$ is monic in $x$ because $Q_{i}$ is a Vaqui\'e key polynomial, then $f_{m}=1$.\\

Now we have $\mu_{i_{0}}(Q_{i}) \leq m\beta_{i_{0}}$. If we have $\mu_{i_{0}}(Q_{i})<m\beta_{i_{0}}$ (where $\beta_{i_0}=\mu(Q_{i_0})=\mu_{i_0}(Q_{i_0})$) then we will have $in_{\mu_{i_0}}Q_i=in_{\mu_{i_0}}(Q_i-Q_{i_{0}}^m$) which contradicts the fact that $Q_{i}$ is $\mu_{i_0}$-minimal.\\
Therefore $\mu_{i_{0}}(Q_{i})=m\beta_{i_{0}}$.\\

We have $Q_{i}=Q_{i_0}^m + f_{m-1}Q_{i_0}^{m-1}+ ... + f_{0}$, with $\mu_{i_{0}}(Q_{i})=m\beta_{i_{0}}=m\mu_{i_{0}}(Q_{i_0})$.

As for all $j$, $0\leq j\leq m-1$, we have $deg_{x}f_j<deg_{x}Q_{i_0}$ hence we can write $f_{j}=\sum\limits_{\gamma_{i_0}}c_{ji\gamma_{i_0}}\textbf{Q}_{i_0}^{\gamma_{i_0}}$\\
where $c_{ji\gamma_{i_0}}\textbf{Q}_{i_0}^{\gamma_{i_0}}$ is an $i_0$-standard monomial not involving $Q_{i_0}$.\\

We have $j\beta_{i_0}+\mu(\sum\limits_{\gamma_{i_0}}c_{ji\gamma_{i_0}}\textbf{Q}_{i_0}^{\gamma_{i_0}})=j\beta_{i_0}+\mu(f_j)\geq  \mu_{i_{0}}(Q_{i})=m\beta_{i_{0}}$.\\

Finally, by definition of $\mu_{i}$, we have $\mu(Q_i)>\mu_{i_0}(Q_i)$,\\
and we have proved that $\mu_{i_0}(Q_{i})=m\beta_{i_{0}}$, then $\mu_{i}(Q_{i})>m\beta_{i_{0}}$.\\
\\
\textbf{Case 2} $i=i^{(t)}\in I^{(t)}$ such that $i\in A^{(t)}$.\\
Pick any $\alpha$ in $A^{(t)}$ such that $\alpha < i$; note that if $i$ is the first element of $A^{(t)}$ we take $\alpha=n^{(t)}$ the final element of the discrete set of the simple admissible family $S^{(t)}$. Take $i_{0}=\alpha$.\\

We know that $\deg_{x}(Q_{i})=\deg_{x}(Q_{i_0})$, therefore we can write
$$
Q_{i}=Q_{i_0}+z_{i_0}
$$
where $z_{i_{0}}\in K[x]$ with $\deg_{x}z_{i_0}< \deg_{x}Q_{i_0}$.\\
We have $\mu_{i_0}(Q_{i})=\min\left\{\mu_{i_0}(Q_{i_0}), \mu_{i_0}(z_{i_0})\right\}$,\\
if $\mu_{i_0}(Q_{i_0})>\mu_{i_0}(Q_{i})$, then $Q_{i}$ is $\mu_{i_0}$-equivalent to $z_{i_0}$, which contradicts the fact that $Q_{i}$ is $\mu_{i_0}$-minimal.\\
Hence $\mu_{i_0}(Q_{i_0})=\mu_{i_0}(Q_{i})$.\\

Moreover, we have $\mu_{i}(Q_{i})>\mu_{i_0}(Q_{i})=\min\left\{\mu_{i_0}(Q_{i_0}), \mu_{i_0}(z_{i_0})\right\}=\min\left\{\mu_{i}(Q_{i_0}), \mu_{i}(z_{i_0})\right\}$, hence
$$
\mu_{i}(Q_{i})>\mu_{i}(Q_{i_0})=\mu_{i}(z_{i_0}).
$$

\textbf{Case 3} $i=1^{(t)}$, and $t\neq1$, this is the case when $Q_{i}$ is a limit key polynomial for the continued family of valuations $(\mu_{i})_{i\in A^{(t-1)}}$.
As the rank of the group $\Gamma$ is equal to $1$, the subset $\Lambda^{t-1}:=\left\{\mu(Q_{\alpha}), \alpha\in A^{(t-1)} \right\}$ does not admit a largest element but an upper bound in $\Gamma$.
By \cite{V0} (Theorem 3.5, page 33) there exists an integer $m$, such that for $\alpha$ sufficiently large in $A^{(t-1)}$, we have:
\begin{equation}\label{limitvaquiepol}
Q_{i}=Q_{\alpha}^m + f_{m-1}Q_{\alpha}^{m-1}+ ... + f_{0}
\end{equation}
with $\mu_{\alpha}(Q_{i})=m\beta_{\alpha}=m\mu_{\alpha}(Q_{\alpha})$.\\
By construction of the family $F$, $Q_{i}$ satisfies $\mu(Q_{i})>\mu_{\alpha}(Q_{i})$ for all $\alpha<i$. We have $\mu_{i}(Q_{i})=\mu(Q_{i})$ by definition of $\mu_{i}$. \\
Then $\mu_{i}(Q_{i})>\mu_{\alpha}(Q_{i})=m\beta_{\alpha}$ for $\alpha$ sufficiently large in $A^{(t-1)}$.\\
Therefore pick $\alpha$ sufficiently large in $A^{(t-1)}$ and take $i_{0}=\alpha$.\\

By the choice of $i_{0}$ and by the first condition of Definition \ref{def_admitted}, $Q_{i}$ has minimal degree among all the polynomials satisfying $\nu_{i_{0}}(Q_{i})<\mu(Q_{i})$.

\end{proof}

\begin{deft} A \textbf{partial collection of HOS key polynomials associated to }$F$ is a pair of sets $(I', J)$ and a collection of polynomials $\{Q'_i\}_{i\in J}$ having the following properties:
\begin{enumerate}
\item $I'\subset I$.
\item $J\supset I'$
\item $I'$ is cofinal in $J$
 \item $\{Q'_i\}_{i\in J}$ is a $J$-th set of HOS key polynomials
 \item for any two consecutive elements  $i_0,i_1\in I'$, there is at most one element $j\in J$ satisfying (\ref{eq:jbetween}). If such a $j$ exists, we have
 $\deg\ Q'_j=\deg\ Q_{i_1}$.
 \item For each $i\in I'$ there exists $i_0\in I'$ with $i=i_0+$ such that the key polynomial
$Q'_i$ satisfies $\mu_{i_0}(Q_{i}-Q'_{i})>\mu(Q_{i})=\mu(Q'_{i})$.
\end{enumerate}
\end{deft}
Next, we introduce the following partial ordering on the set of all the partial collections of HOS key polynomials, associated to $F$.
\begin{deft} Let $(I',J)$, with $\{Q'_i\}_{i\in J}$ the corresponding $J$-th set of HOS key polynomials and $(I'',J')$, with $\{Q''_j\}_{j\in J'}$ the corresponding $J'$-th set of HOS key polynomials, be two partial collections of HOS key polynomials, associated to $F$. We say that $(I',J)\preceq (I'',J')$ if there is an inclusion $I'\subset I''$, and an inclusion $J\subset J'$, such that $Q'_i=Q''_i$ for all $i\in J$.
\end{deft}

\begin{lem}
The set of all the partial collections of HOS key polynomials associated to $F$ is not empty.
\end{lem}
\begin{proof}
Let $I'=\{1^{(1)}\}$, put $Q'_{1^{(1)}}=Q_{1^{(1)}}$ and let $J=\left\{ 1^{(1)} \right\}$.
\end{proof}

The partially ordered set of all the partial collections of HOS key polynomials, associated to $F$, satisfies the hypotheses of Zorn's lemma, and therefore contains a maximal element. \\

Therefore, to prove Theorem \ref{VaquietoHOS}, it remains to prove the following statement: if $(I', J)$ is a partial collection of HOS key polynomials, associated to $F$, such that $I'$ is not cofinal in $I$ then there exists a partial collection $(I'', J')$ of HOS key polynomials, associated to $F$, such that $(I', J)\prec(I'', J')$.\\

Let $(I' ,J)$ be a partial collection of HOS key polynomials associated to $F$, such that
$I'$ is not cofinal in $I$.\\

To prove the above statement, we will define a partial collection $(I'', J')$ of HOS key polynomials, associated to $F$, such that $(I', J)\prec(I'', J')$.\\

As $I'$ is not cofinal in $I$, there exists $\alpha\in I$ such that $\alpha>I'$. As $I'$ is cofinal in $J$ we have $\alpha>J$.\\

We have two cases:
\begin{enumerate}
\item $J$ admits a maximal element $i_0$.

\item $J$ does not admit a maximal element, but there exists a subset $E\subset I$ such that $E=I^{(1)}\bigcup...I^{(t_{0}-1)}\bigcup I^{(t_{0})}$
with $1\leq t_0<N$ and $I^{(t_0)}=\left\{ 1^{(t_0)},...,n^{(t_0)}\right\}\bigcup A^{(t_0)}$ with $A^{(t_0)}$ not empty, $I'\subset E$ and $I'$ is cofinal in $E$.
\end{enumerate}

We want to define a set $J'$ such that $J\subsetneqq J'$.\\

If $J$ admits a maximal element $i_0$, as $J$ is cofinal in $I'$ and $I'\subset I$, we have $i_0\in I$.\\
Hence there exists $t$, $1 \leq t<N$ such that  $i_{0}\in I^{(t)}=\left\{ 1^{(t)},...,n^{(t)}\right\}\bigcup A^{(t)}$.\\

If $i_0=s^{(t)}$ with $s^{(t)}<n^{(t)}$, put $i$=$(s+1)^{(t)}$.

If $i_0=n^{(t)}$ or $i_0\in A^{(t)}$, then choose any $\alpha\in A^{(t)}$ such that $\alpha>i_0$ and put
$i=\alpha$.\\

If $J$ does not have a maximal element, take $E$ as above, and put $i=1^{(t_0+1)}$.\\

We will define the polynomial $Q'_{i}$.

\begin{deft}
Let $t$, $2\leq t<N$. For an element $i$, $i\in I$, we say that $i$ is a limit ordinal for the family $F$ if $i=1^{(t)}$. Otherwise we say that $i$ is a simple ordinal for $F$.
\end{deft}

Assume that $i$ is a simple ordinal for $F$.

\begin{deft}\label{simpleordinalqi}
If $i\in \{ 1^{(t)},...,n^{(t)}\}$ with $1\leq t<N$,
write $Q_{i}$ as in (\ref{eq:standardform}).\\

For $j\in\{0,\dots,\alpha_i-1\}$, let $a_{ji_0}$ denote the sum of all the monomials $c_{ji\gamma_{i_0}}\textbf{Q}_{i_0}^{\gamma_{i_0}}$ of value $\mu_{i_0}(Q_{i})-j\beta_{i_0}=\alpha_{i}\beta_{i_0}-j\beta_{i_0}$ (if the set of such monomials is empty, we put $a_{ji_0}=0$). Put $P_{i}=Q_{i_0}^{\alpha_{i}}+\sum\limits_{j=0}^{\alpha_{i}-1}a_{ji_{0}}Q_{i_0}^j$.
If $i\in A^{(t)}$ with $1\leq t<N$ we put $P_{i}=Q_{i}$.
\end{deft}

\begin{prop}
$P_{i}$ is a Vaqui\'e key polynomial for the valuation $\mu_{i_0}$.
\end{prop}

\begin{proof}
Indeed, $P_{i}$ is monic, and $P_{i}$ and $Q_{i}$ are $\mu_{i_0}$-equivalent by definition.
\end{proof}

We will now define the polynomial $P_{i}$ in the case when $i$ is a limit ordinal.\\

Let $i=1^{(t_{0}+1)}$, $Q_{i}$ is a limit key polynomial in the sense of Vaqui\'e for the continued family of valuations $\left\{\mu_{i'}\right\}_{i'\in A^{(t_{0})}}$.

As the rank of the group $\Gamma$ is equal to $1$, the subset $\Lambda^{(t_{0})}:=\left\{\mu(Q_{\alpha}), \alpha\in A^{(t_{0})} \right\}$ does not admit a maximal element but an upper bound $\overline{\beta_{t_{0}}}$ in $\Gamma$.
\\

By the proof of Proposition 55 (\cite{HOS} page 1063), if $f=\sum\limits_{j=0}^{s}a_{ji'}Q_{i'}^{j}\in K[x]$ with $i'\in J$, satisfies $\mu_{i''}(f)<\mu(f)$ for all $i''\in J$, then there exists an $i_{1}\in J$ and an integer $0<j \le s$, $j=p^{e_{0}}$ for $e_{0}$ a strictly positive integer, such that for all $i'>i_{1}\in J$ the polynomial $f':=b_{p^{e_{0}}i_{1}}Q_{i_{1}}^{p^{e_{0}}} + \sum\limits_{j=0}^{p^{e_{0}}-1}b_{ji_{1}}Q_{i_{1}}^{j}$ (where  $b_{ji_{1}}$ is an $i_{1}$-standard expansion not involving $Q_{i_{1}}$ for all $0\leq j< p^{e_{0}}$) satisfies $\mu_{i''}(f')<\mu(f')$ for all $i''\in J'$. The integer $p^{e_{0}}$ is the minimal degree in $Q_{i'}$ with $i'\in J$, of a polynomial $f\in K[x]$ satisfying $\mu_{i''}(f)<\mu(f)$ for all $i''\in J$.\\
\\
By \cite{V0} (Theorem 3.5, page 33) there exists $i_{0}\in J$ such that for all $i'>i_{0}\in J$, we can write $Q_{i}=Q_{i'}^{m} + \sum\limits_{j=0}^{m-1}d_{ji'}Q_{i'}^{j}$ with $d_{ji'}\in K[x]$, $\deg_xd_{ji'}<\deg_xQ_{i'}$, and $\mu_{i'}(Q_{i})=m\beta_{i'}=\mu(d_{0i'})$.\\
\\

By the first condition of Definition \ref{def_admitted} we know that $Q_{i}$ is the polynomial with the minimal degree among those which satisfy $\mu_{i'}(Q_{i})<\mu(Q_{i})$ for all $i'\in J'$. Hence $m=p^{e_{0}}$.\\
\\

Write $Q_{i}=Q_{i_{0}}^{m} + \sum\limits_{j=0}^{m-1}a_{ji_{0}}Q_{i_{0}}^{j}$ with $a_{ji_{0}}$ and $i_{0}$-standard expansion not involving $Q_{i_{0}}$, then also by the proof of Proposition 55 in \cite{HOS} there exists $i_{1}>i_{0}$ and a polynomial $P_{i}=Q_{i_{1}}^{p^{e_{0}}} + a_{0i_{1}} + \sum\limits_{j=1}^{p^{e_{0}}-1}a'_{ji_{1}}Q_{i_{1}}^{j}$ where $a'_{ji_{1}}=a_{ji_{1}}$ or $a'_{ji_{1}}=0$ and $P_{i}$ is a weakly affine $i_{1}$-standard expansion with $\mu_{i_{1}}(P_{i})=p^{e_{0}}\beta_{i_{1}}=\mu(a_{0i_{1}})$
and for all $0<j\le p^{e_{0}-1}$ we have $\mu(a'_{ji_{1}})+j\overline{\beta}=p^{e_{0}}\overline{\beta}$.

As well, $P_{i}$ satisfies $\mu(P_{i})>\mu_{i'}(P_{i})$ for all $i'\in J$. Hence $P_{i}$ is a limit Vaqui\'e key polynomial for the family $\left\{\mu_{\alpha}\right\}_{\alpha\in A^{(t_{0})}}$.
\\
For simplicity, we will replace $i_1$ by $i_0$ in the above definition of $P_i$.

Now we will define the valuation $\mu'_{i}$.

\begin{deft}
Put $\beta'_{i}=\mu(P_{i})$.
If $i$ is a simple ordinal, we define the valuation $\mu'_{i}:=[\mu_{i_{0}}; \mu(P_{i})=\beta'_{i}]$.
If $i$ is a limit ordinal $i=1^{(t_0+1)}$, we define the valuation $\mu'_{i}:=[\{\mu_{\alpha}\}_{\alpha\in A^{(t_0)}}; \mu(P_{i})=\beta'_{i}]$.
\end{deft}

As $P_{i}$ and $\mu'_{i}$ are well defined, we will study the valuation $\mu'_{i}$.

\begin{rek}\label{samevaluation}
Put $h=Q_{i}-P_{i}$.
We have $deg_{x}(h)<deg_{x}Q_{i}=deg_{x}P_{i}$, therefore $\mu_{i_{0}}(h)=\mu(h)$.\\
\\
Furthermore, if \begin{equation}\label{restrictioncondition}
\mu(h)\geq \beta'_{i}=\beta_{i},
\end{equation}
then by Proposition \ref{sameaugmentedvaluation} and Proposition \ref{samelimitaugmentedvaluation} we have $\mu'_{i}=\mu_{i}$.
\end{rek}

\begin{prop}\label{muprimeiequalmui}
If $\mu_{i}(P_{i})=\beta'_{i}$ then we have:
\begin{enumerate}
\item $\mu(h)\geq \beta'_{i}$.
\item $\beta_{i}\geq \beta'_{i}$.
\end{enumerate}
\end{prop}

\begin{proof}
\begin{enumerate}
\item We have $Q_{i}=P_{i}+h$, with $\mu_{i_0}(h)=\mu(h)$ then
\begin{equation}\label{betai1}
\beta_{i}=\mu(Q_{i})\geq \min\left\{ \mu(P_{i}), \mu(h) \right\}=\min\left\{ \beta'_{i}, \mu(h) \right\}.
\end{equation}

On the other hand, by definition of $\mu_{i}$ we have
\begin{equation}\label{betai2}
\beta'_{i}=\mu_{i}(P_{i})=\min\left\{ \beta_{i}, \mu(h) \right\}.
\end{equation}

Suppose that $\mu(h)<\beta'_{i}$, then from (\ref{betai1}) $\beta_{i}\geq \mu(h)$ then from (\ref{betai2})
$\beta'_{i}=\mu(h)$ which is impossible. Therefore we have $\mu(h)\geq \beta'_{i}$.

\item Now 2 follows from (\ref{betai1}).
\end{enumerate}

\end{proof}

Now we can construct the partial collection of HOS key polynomials $(I'', J')$.\\

If $\mu_{i}=\mu$, we have $\mu_{i}(P_{i})=\mu(P_{i})=\beta'_{i}$ and by Proposition \ref{muprimeiequalmui}
we have $\mu(h)\geq\beta'_{i}$ and $\beta_{i}\geq\beta'_{i}$.

If $\mu(h)>\beta'_{i}$, or if $\mu(h)=\beta'_{i}$ and $\beta_{i}=\beta'_{i}$ then the condition (\ref{restrictioncondition}) is satisfied and by Remark \ref{samevaluation} we have $\mu'_{i}=\mu_{i}=\mu$.

Let $l$ be a new index. We put $I''=I'\bigcup \{l \}$ and $J'=J\bigcup \{l \}$, $Q'_l=P_i$. The set $\{Q'_{j}\}_{j\in J'}$ is a $J'$-th set of HOS key polynomials. $(I'', J')$ is a partial collection of HOS key polynomials associated to $F$ with $(I', J)\prec (I'', J')$.\\

If $\mu(h)=\beta'_{i}$ and $\beta_{i}>\beta'_{i}$,
then we have $Q_{i}=P_{i}+h$ with $deg_{x}h<deg_{x}P_{i}$ and $\beta_{i}>\beta'_{i}=\mu(h)$.\\
We define two new indices $i_{1}, l$ such that $J<i_{1}<l$. We put $Q'_{i_{1}}=P_{i}$ and $Q'_{l}=Q_{i}$.
We put $I''=I'\bigcup \{l\}$ and $J'=J\bigcup \{i_{1},l\}$; the set $\{Q'_{j}\}_{j\in J'}$ is a $J'$-th set of HOS key polynomials. Then $(I'', J')$ is a partial collection of HOS key polynomials associated to $F$ with $(I', J)\prec (I'', J)$.\\

From now on assume that $\mu_{i}\neq\mu$.\\

If $i\in I^{(t)}$ such that $i=s^{(t)}$ with $s^{(t)}<n^{(t)}$, then $i+1$ exists in $I$ and $i+1=(i+1)^{(t)}$ and the polynomial $f$ with the minimal degree that satisfies $\mu(f)>\mu_{i}(f)$ has degree $deg_{Q_{i}}Q_{i+1}>1=deg_{Q_{i}}P_{i}$, hence the polynomial $P_{i}$ satisfies $\mu_{i}(P_{i})=\mu(P_{i})$ and by Proposition \ref{muprimeiequalmui} we have $\mu(h)\geq\beta'_{i}$ and $\beta_{i}\geq\beta'_{i}$.

If $\mu(h)>\beta'_{i}$, or if $\mu(h)=\beta'_{i}$ and $\beta_{i}=\beta'_{i}$ then the condition (\ref{restrictioncondition}) is satisfied and by Remark \ref{samevaluation} we have $\mu'_{i}=\mu_{i}$.\\

In this case, we put $I''=I'\bigcup \{i \}$ and $J'=J\bigcup \{i \}$, $Q'_i=P_i$. The set $\{Q'_{j}\}_{j\in J'}$ is a $J'$-th set of HOS key polynomials. Then $(I'', J')$ is a partial collection of HOS key polynomials associated to $F$ with $(I', J)\prec(I'', J)$.

If $\mu(h)=\beta'_{i}$ and $\beta_{i}>\beta'_{i}$, then we have $Q_{i}=P_{i}+h$ with $deg_{x}h<deg_{x}P_{i}$ and $\beta_{i}>\beta'_{i}=\mu(h)$.\\
We define two new indices $i_{1}, l$ such that $J<i_{1}<l$. We put $Q'_{i_{1}}=P_{i}$ and $Q'_{l}=Q_{i}$.
We put $I''=I'\bigcup \{l\}$ and $J'=J\bigcup \{i_{1},l\}$; the set $\{Q'_{j}\}_{j\in J'}$ is a $J'$-th set of HOS key polynomials. Then $(I'', J')$ is a partial collection of HOS key polynomials associated to $F$ with $(I', J)\prec(I'', J)$.\\

If $i=n^{(t)}$ we have two cases:\\

Case 1: $P_{i}$ satisfies the condition (\ref{restrictioncondition}), we have $\mu'_{i}=\mu_{i}$.
We choose a cofinal subset $D^{(t)}$ in $A^{(t)}$ which is of order type $\Bbb N$.

We put $I''=I'\bigcup \{i\}\bigcup \{D^{(t)}\}$ and $J'=J\bigcup \{i\}\bigcup \{D^{(t)}\}$, $Q'_i=P_i$, and we take the set $\{Q'_{i'}\}_{i'\in J''}$  with $Q'_{i'}=Q_{i'}$ if $i'\in \{D^{(t)}\}$. The set $\{Q'_{j}\}_{j\in J'}$ is a $J'$-th set of HOS key polynomials and $(I'', J')$ is a partial collection of HOS key polynomials associated to $F$ with $(I', J)\prec (I'', J')$.

Case 2: $P_{i}$ does not satisfy the condition (\ref{restrictioncondition}), then either
$\mu(P_{i})>\mu(Q_{i})$ or $\mu(P_{i})<\mu(Q_{i})$.\\

Suppose that $\mu(P_{i})>\mu(Q_{i})$. As the family $\{\mu_{\alpha}\}_{\alpha\in A}$ is exhaustive, there exists
a polynomial $Q_{\alpha}$ with $\alpha\in A$ such that $\mu(P_{i})=\mu(Q_{\alpha})$. In this case
we put $Q'_{i}=Q_{\alpha}$ witch is also a Vaqui\'e key polynomial for the valuation $\mu'_{i_0}=\mu_{i-1}$.\\
We put :\\
$A'^{(t)}=\left\{\left.\alpha'\in A^{(t)}\ \right|\ \alpha'>\alpha\right\}$,\\
We choose a cofinal subset $D^{(t)}$ from $A'^{(t)}$ of order type $\Bbb N$.
We put $I''=I'\bigcup \{\alpha\} \bigcup D^{(t)}$ and $J'=J\bigcup \{\alpha\} \bigcup D^{(t)}$,\\
and we take the set $\{Q'_{i'}\}_{i'\in J'}$  with $Q'_{i'}=Q_{i'}$ if $i'\in \{D^{(t)}\}$. The set $\{Q'_{i'}\}_{i'\in J'}$ is a $J'$-th set of HOS key polynomials. Then $(I'', J')$ is a partial collection of HOS key polynomials associated to $F$ with $(I', J)\prec(I'', J')$.

Now assume that $\mu(P_{i})<\mu(Q_{i})$. Then we have $Q_{i}=P_{i}+h$ with $deg_{x}h<deg_{x}P_{i}$ and $\beta_{i}>\beta'_{i}=\mu(h)$.\\
We define two new indices $i_{1}, l$ such that $J<i_{1}<l$. We put $Q'_{i_{1}}=P_{i}$ and $Q'_{l}=Q_{i}$.
We choose a cofinal subset $D^{(t)}$ from $A^{(t)}$ of order type $\Bbb N$.
We put $I''=I'\bigcup \{l\}\bigcup \{D^{(t)}\}$ and $J'=J\bigcup \{i_{1},l\}\bigcup \{D^{(t)}\}$ the set $\{Q'_{j}\}_{j\in J'}$ is a $J'$-th set of HOS key polynomials.
Then $(I'', J')$ is a partial collection of HOS key polynomials associated to $F$ with $(I', J)\prec (I'', J')$.

The only case that left is the case when $i=\alpha$ with $\alpha\in A^{(t)}$. In this case
we have $P_{i}=Q_{\alpha}$ and $\mu'_{i}=\mu_{\alpha}$.\\
We choose a cofinal subset $D^{(t)}$ from $A^{(t)}$ of order type $\Bbb N$ such that $D^{(t)}>\alpha$.
We put $I''=I'\bigcup \{i\}\bigcup \{D^{(t)}\}$ and $J'=J\bigcup \{i\}\bigcup \{D^{(t)}\}$ and we take the set $\{Q'_{i'}\}_{i'\in J'}$ with $Q'_{i'}=Q_{i'}$ if $i'\in \{D^{(t)}\}$.  The set $\{Q'_{i'}\}_{i'\in J'}$ is a $J'$-th set of HOS key polynomials. Then $(I'', J')$ is a partial collection of HOS key polynomials associated to $F$ with $(I', J)\prec(I'', J')$.

We have constructed a partial collection of HOS key polynomials associated to $F$ with $(I', J)\prec(I'', J')$.\\

By Zorn's Lemma, the partially ordered set of all the partial collections of HOS key polynomials, associated to $F$ contains a maximal element, hence there exists a partial collection $(I^{e}, J^{e})$ of HOS key polynomials associated to $F$ witch is maximal for the partial ordering of the set of all the partial collections of HOS key polynomials associated to $F$.\\
Let $\{Q'_{i}\}_{i\in J^{e}}$ be the $J^{e}$-th set of HOS key polynomials associated to $(I^{e}, J^{e})$.\\

From the proof above, $I^{e}$ is cofinal in $I$.\\

Now suppose that the family $F$ converges to $\mu$.
then for all $f\in K[x]$, there exists $i\in I^{(e)}$ such that $\mu^{e}_{i}(f)=\mu(f)$.
As $I^{e}$ is cofinal in $I$,
then for all $f\in K[x]$, there exists $i\in I^{(e)}$ such that $\mu^{e}_{i}(f)=\mu(f)$.
And as $I^{e}\subset J^{e}$
therefore the set $\{Q'_j\}_{j\in J^{e}}$ is a complete set of HOS key polynomials for $\mu$.

\end{proof}

\section{Example.}\label{example_limit}

In this section, we want to give an example of a limit key polynomial, such that both valuations $\nu$ and $\mu$ are centered in local noetherian rings, one of which dominates the other.\\

We start by giving some definitions and some properties of key polynomials and augmented valuations.\\

Let $x$ be a variable and $k$ a field. Let $\nu$  be a valuation of $k[x]$, $\Gamma$ an ordered group containing $\nu(K[x])$ as a sub-group, $\phi$ a key polynomial for $\nu$, and $\gamma\in\Gamma$ such that $\gamma>\nu(\phi)$. \\
Let $\mu=[\nu;\mu(\phi)=\gamma]$. For all $f=\sum\limits_{j=0}^{s}a_{j}\phi^{j}\in K[x]$, with $a_{j}\in k[x]$ and $deg_{x}a_{j}<deg_{x}\phi$ for $0\leq j\leq s$, we define $D_{\phi}(f)=max\left\{j\in\left\{0,...,s\right\}\ /\ \mu(f)=\nu(a_{j})+ j\gamma\right\}$.\\
If $d=D_{\phi}(f)$ then by definition we have $\init_{\mu}f=\init_{\mu}\sum\limits_{j=0}^{d}a_{j}\phi^{j}$.\\
We notice that the integer $D_{\phi}(f)$ depends only on the image $\init_{\mu}f$ in the graded algebra $G_{\mu}$,
therefore if $f$ and $f'$ are $\mu$-equivalent, then $D_{\phi}(f)=D_{\phi}(f')$.\\
We have also $D_{\phi}(f.g)=D_{\phi}(f)+D_{\phi}(g)$.
\begin{lem}\label{minimal}
Let $f=\sum\limits_{j=0}^{s}a_{j}\phi^{j}$, with $a_{j}\in k[x]$ and $deg_{x}a_{j}<deg_{x}\phi$ for $0\leq j\leq s-1$ and $deg_{x}a_{s}=0$
then:
 $$
 \mu(a_{s}\phi^{s})=\mu(f)\Rightarrow\ f\ is\ \mu-minimal.
 $$
\end{lem}
\begin{proof}
Let $g\in k[x]$ such that $f\ \mu$- divides $g$.\\
Then there exists $h\in k[x]$ such that
$\init_{\mu}g=\init_{\mu}h\init_{\mu}f$, therefore $D_{\phi}g\geq D_{\phi}f$,\\
and for all $g\in k[x]$ we have $deg_{x}g\geq D_{\phi}g.deg_{x}\phi$.\\
On the other hand from the hypothesis we have $deg_{x}f=D_{\phi}f.deg_{x}\phi$.\\
Finally, we find $deg_{x}g \geq D_{\phi}g.deg_{x}\phi \geq D_{\phi}f.deg_{x}\phi = deg_{x}f$.

\end{proof}

Let $k$ be a field of characteristic $p>2$.\\
Consider the pure transcendental field extension $k(z)$ of $k$ with the $z$-adic valuation $\nu_{0}$ (with $\nu_{0}(z)=1$).\\

Consider the field $K=k(y,z)$ where $y$ is a pure transcendental element over $k(z)$. \\
We define the valuation $\nu$ of $K$ which extends $\nu_{0}$ in the following way:\\

Put : \\
\begin{flalign*}
Q_{y,1}&=y& \gamma_{y,1}&=\frac{1}{2}&\\
Q_{y,2}&=y^2+z& \gamma_{y,2}&=p-\frac{1}{4}&\\
Q_{y,3}&=Q_{y,2}^{2}+z^{2p-1}y& \gamma_{y,3}&=2p-\frac{1}{8p}&\\
Q_{y,4}&=Q_{y,3}^{2p}-z^{4p^{2}-p}Q_{y,2}& \gamma_{y,4}&=4p^2-\frac{1}{16}&\\
Q_{y,5}&=Q_{y,4}^{2}+z^{6p^{2}}Q_{y,3}^{p}& \gamma_{y,5}&=8p^{2}-\frac{1}{32p}&\\
\end{flalign*}
For $j>2$ put:
\begin{flalign*}
Q_{y,2j}&=Q_{y,2j-1}^{2p}+z^{2^{2j-2}p^{j}-2^{2j-4}p^{j-1}}Q_{y,2j-2}& \gamma_{y,2j}&=2^{2j-2}p^{j}-\frac{1}{2^{2j}}&\\
Q_{y,2j+1}&=Q_{y,2j}^{2}+z^{3(2)^{2j-3}p^{j}}Q_{y,2j-1}^{p}& \gamma_{y,2j+1}&=2^{2j-1}p^{j}-\frac{1}{2^{2j+1}p}&\\
\end{flalign*}

Now we define recursively for all $j\geq 1$ the augmented valuation $\nu_{j}=[\nu_{j-1};\nu_{j}(Q_{y,j})=\gamma_{j}]$. In fact, the construction of $\nu_{1}$ is obvious. And we notice that every polynomial $Q_{y,j}$ is a key polynomial for the valuation $\nu_{j-1}$.

Indeed, each polynomial $Q_{y,j}$ is monic, and by lemma \ref{minimal}, for all $j\geq 1$, $Q_{y,j}$ is $\nu_{j-1}$-minimal.\\

Now to prove that $Q_{y,2j}$ is $\nu_{2j-1}$-irreducible, it is sufficient
to prove that the image of the monomial $z^{2^{2j-2}p^{j}-2^{2j-4}p^{j-1}}Q_{y,2j-2}$, in the
graded algebra $G_{\nu_{2j-1}}$ is neither a square $2$ nor a $p$-th power,
i.e it is sufficient to prove that $\nu_{2j-1}(z^{2^{2j-2}p^{j}-2^{2j-4}p^{j-1}}Q_{y,2j-2})$
is not divisible by either $2$ or $p$ in the group $\mathbb{Z}+\nu(Q_{y,1})\mathbb{Z}+...+\nu(Q_{y,2j-1})\mathbb{Z}$.
As $in_{\nu_{2j-1}}Q_{y,2j-1}$ is transcendental over $in_{\nu_{2j-1}}\bold Q_{y,2j-1}$ in $G_{\nu_{2j-1}}$,
it is sufficient to prove that $\nu_{2j-1}(z^{2^{2j-2}p^{j}-2^{2j-4}p^{j-1}}Q_{y,2j-2})$ is not divisible by either $2$ or $p$ in the group $\mathbb{Z}+\nu(Q_{y,1})\mathbb{Z}+...+\nu(Q_{y,2j-2})\mathbb{Z}$.\\
Now $\nu_{2j-1}(z^{2^{2j-2}p^{j}-2^{2j-4}p^{j-1}}Q_{y,2j-2})=2^{2j-2}p^{j}-\frac{1}{2^{2j-2}}=2p(2^{2j-3}p^{j-1}-\frac{1}{2^{2j-1}p})$.
But neither $2(2^{2j-3}p^{j-1}-\frac{1}{2^{2j-1}p})$ nor $p(2^{2j-3}p^{j-1}-\frac{1}{2^{2j-1}p})$ can be
in $\mathbb{Z}+\nu(Q_{y,1})\mathbb{Z}+...+\nu(Q_{y,2j-2})\mathbb{Z}$.
Hence $Q_{y,2j}$ is $\nu_{2j-1}$-irreducible.

And with the same method we prove that $Q_{y,2j+1}$ is $\nu_{2j}$-irreducible.

We notice that for all $f\in k(z)[y]$, there exists $j\in \mathbb{N}$ such that $\forall i>j$, $\nu_{i}(f)=\nu_{j}(f)$. Therefore we can define the valuation $\nu$ by:
$$
\forall f\in k(z)[y],\ \nu(f)=max\left\{\nu_{j}(f)\ |\ j\in\mathbb{N}\right\}
$$

Now, consider the field $K(x)$ with $x$ a pure transcendental element over $K$.

We define the valuation $\mu$ of $K(x)$ which extends the valuation $\nu$ in the following way: \\

We will first define $h_{i}(y,z)\in K$. Let $h_{i}$ be defined by:

\begin{flalign*}
h_{1}&=\frac{Q_{y,3}^2}{z^{4p-1}}& \nu(h_{1})=2\nu(Q_{y,3})-\nu(z^{4p-1})=4p-\frac{1}{4p}-4p+1=1-\frac{1}{4p}\\
h_{i}&=\frac{Q_{y,2i+1}^2}{z^{2^{2i}p^{i}-1}}& \nu(h_{i})=2\nu(Q_{y,2i+1})-\nu(z^{2^{2i}p^{i}-1})=2^{2i}p^{i}-\frac{1}{2^{2i}p}-2^{2i}p^{i}+1=1-\frac{1}{2^{2i}p}\\
\end{flalign*}

Now put : $\mu(x)=1-\frac{1}{4p}$.
We have $\mu(x)=\mu(h_{1})$, and the value of the polynomial $x-h_{1}$ is not determined by the values of $x$ and $h_1$. Put
\begin{flalign*}
	Q_{x,1}&=x-h_{1}& and& &\mu(Q_{x,1})=1-\frac{1}{2^{4}p}>\mu(x).&\\
\end{flalign*}
we have $\mu(Q_{x,1})=\mu(h_{2})$, and the value of the polynomial $x-h_{1}-h_{2}=Q_{x,1}-h_{2}$ is not determined by the values of $Q_{x,1}$ and $h_2$. Put
\begin{flalign*}
Q_{x,2}&=x-h_{1}-h_{2}& and& &\mu(Q_{x,2})=1-\frac{1}{2^{6}p}>\mu_{1}(Q_{x,2})=\mu(Q_{x,1})&\\
\end{flalign*}
where $\mu_{1}$ is the $i$-truncation associated to the key polynomial $Q_{x,1}$.\\
We have $\mu(Q_{x,2})=\mu(h_{3})$, and the value of the polynomial $x-h_{1}-h_{2}-h_{3}=Q_{x,2}-h_{3}$ is not determined by the values of $Q_{x,2}$ and $h_3$. Put
\begin{flalign*}
Q_{x,3}&=x-h_{1}-h_{2}-h_{3}& and& &\mu(Q_{x,3})=1-\frac{1}{2^{8}p}>\mu_{2}(Q_{x,3})=\mu(Q_{x,2})&\\
\end{flalign*}
where $\mu_{2}$ is the $i$-truncation associated to the key polynomial $Q_{x,2}$.\\

We can construct by induction, an infinite family of key polynomials $\left\{Q_{x,i}\right\}_{i\in \mathbb{N}}$, with an infinite family of $i$-truncations $\left\{\mu_{i}\right\}_{i\in \mathbb{N}}$ associated to $\left\{Q_{x,i}\right\}_{i\in \mathbb{N}}$.
\begin{flalign*}
Q_{x,i}&=x-h_{1}-h_{2}-...-h_{i}& and& &\mu(Q_{x,i})=1-\frac{1}{2^{2i+2}p}.&\\
\end{flalign*}

To simplify the notation we will denote $Q_{i}:=Q_{x,i}$, and $\beta_{i}=\mu(Q_{i})=\mu_{i}(Q_{i})$.\\

For each $i \in\mathbb{N}$, we have $Q_{x,i}=Q_{x,i-1}-h_{i}$ and $\beta_{i}=\mu_{i}(Q_{i})>\mu_{i-1}(Q_{i})=\mu_{i-1}(Q_{i-1})=\beta_{i-1}$.\\

The sequence $\left\{\beta_{i}\right\}_{i \in\mathbb{N}}$ is strictly increasing and bounded; it does not contain a maximal element. We have $\lim\limits_{i\rightarrow\infty}\beta_{i}=\bar{\beta}=1$.

Now take the polynomial $f=x^{p}-y^2-z$. We want to prove that $f$ is a limit key polynomial, that is, that $f$ is the polynomial of smallest degree which satisfies $\mu(f)>\mu_{i}(f)$ for all $i \in\mathbb{N}$.

Replacing $x$ by $Q_{i-1}=x-h_{1}-h_{2}-...-h_{i}$ in $f$ we find :
$$
f=Q_{i-1}^p+h_{1}^{p}+h_{2}^{p}+...+h_{i-1}^{p}-y^2-z.
$$

We notice that:
$$
-y^2-z+h_{1}^p=-Q_{y,2}+\frac{Q_{y,3}^{2p}}{z^{4p^{2}-p}}=\frac{Q_{y,4}}{z^{4p^{2}-p}}
$$
and
$$
-y^2-z+h_{1}^p+h_{2}^p=\frac{Q_{y,4}}{z^{4p^{2}-p}}+\frac{Q_{y,5}^{2p}}{z^{16p^{3}-p}}=\frac{Q_{y,6}}{z^{16p^{3}-p}}.
$$
It is not hard to prove that for every $i \in\mathbb{N}$ we have
$$
-y^2-z+h_{1}^p+h_{2}^p+...+h_{i-1}^p=\frac{Q_{y,2i}}{z^{2^{2i-2}p^{i}-p}}.
$$

Therefore for all $i\in \mathbb{N}$ we have:
\begin{flalign*}
f&=Q_{i-1}^p+\frac{Q_{y,2i}}{z^{2^{2i-2}p^{i}-p}}&\\
\end{flalign*}
with $\nu(\frac{Q_{y,2i}}{z^{2^{2i-2}p^{i}-p}})=\nu(Q_{y,2i})-(2^{2i-2}p^{i}-p)=2^{2i-2}p^{i}-\frac{1}{2^{2i}}-2^{2i-2}p^{i}+p=p-\frac{1}{2^{2i}}=p\beta_{i-1}$.

Therefore for all $i\in \mathbb{N}$ we have $\mu_{i+1}(f)>\mu_i(f)=p-\frac{1}{2^{2i+2}}$. Put $\mu(f)=p>p-\frac{1}{2^{2i+2}}=\mu_{i}(f)$ for all $i$.

Suppose that there exists $g\in K[X]$, such that $\deg_{x}g<\deg_{x}f=p$ and that $\mu(g)>\mu_{i}(g)$ for all $i \in\mathbb{N}$. We may assume that $g$ is monic, and that $m=\deg_{x}g$ is the minimal degree for all the polynomials $\phi\in K[X]$ that satisfy the relation $\mu(\phi)>\mu_{i}(\phi)$ for all $i \in\mathbb{N}$.

Then there exists $i_{0}$ such that for all $i\geq i_{0}$ in $\mathbb{N}$, we have :
$$
g=Q_{i}^m + g_{m-1}Q_{i}^{m-1}+ ... + g_{0}
$$
with $\mu_{i}(g)=m\beta_{i}=m\mu_{i}(Q_{i})$.\\
As $\deg_{x}g<\deg_{x}f$, write $f=f_{r}g^r + f_{r-1}g^{r-1}+ ... + f_{0}$, with $f_{j}\in K[X]$ with $deg_{x}f_{j}<deg_{x}g$ for all $0\leq j\leq r$.\\
For all $j$, $0\leq j\leq r$, there exists an $i_{1,j}$ such that for all $i\geq i_{1,j}$,
$$
\mu_{i}(f_{j})=\mu_{i+1}(f_{j})=...=\mu(f_{j})=\delta_{j}.
$$
Put $i_{2}=max\left\{i_{0},i_{1,0},...,i_{1,j},...,i_{1,r}\right\}$, then for all $i\geq i_{2}$ we have :
$$
\mu_{i}(f)\geq \min_{0\leq j\leq r}\left\{\mu_{i}(f_{j}g^{j})\right\}=\min_{0\leq j\leq r}\left\{\delta_{j}+jm\beta_{i}\right\}.
$$
The set $\left\{\beta_{i}\ /\ i\geq i_{2} \right\}$ is infinite, and $j$ and $\delta_{j}$
cannot take but a finite number of values, therefore there exists an $i_{3}\geq i_{2}$, such that for all $i\geq i_{3}$, $\min_{0\leq j\leq r}\left\{\delta_{j}+jm\beta_{i}\right\}$ is attained only once, therefore $\mu_{i}(f)=\delta_{j}+jm\beta_{i}$. On the other hand we have $\mu_{i}(f)=p\beta_{i}$ hence
considering $i,i'$ $> i_{3}$ will give :
$\mu_{i}(f)=\delta_{j} + jm\beta_{i}=p\beta_{i}$ and $\mu_{i'}(f)=\delta_{j} + jm\beta_{i'}=p\beta_{i'}$, therefore
substracting these two equations will give $p=jm$. But $p$ is irreducible and $j \leq r< p$ and $m < p$ which is impossible.


\begin{thebibliography}{99}

\bibitem{AM1} S. Abhyankar and T.T. Moh, {\em Newton--Puiseux expansion and generalized Tschirnhausen
transformation I}, Reine Agew. Math., vol 260 (1973), pages 47--83.
\bibitem{AM2} S. Abhyankar and T.T. Moh, {\em Newton--Puiseux expansion and generalized Tschirnhausen
transformation II}, Reine Agew. Math., vol 261 (1973), pages 29--54.
\bibitem{HOS} F. J. Herrera Govantes, M. A. Olalla Acosta, M. Spivakovsky,
{\em Valuations in algebraic field extensions}, Journal of Algebra,
Volume 312, Issue 2 (2007), pages 1033-1074.
\bibitem{L} M. Lejeune-Jalabert, {\em Contributions \`a l'\'etude des
singularit\'es du point de vue du polygone de Newton}, Th\'{e}se
d'Etat, Universit\'{e} Paris 7, (1973).
\bibitem{M1} S. MacLane, {\em A construction for prime ideals as absolute
values of an algebraic field}, Duke Math. J., vol 2 (1936), pages 492--510.
\bibitem{M2} S. MacLane {\em A construction for absolute values in polynomial
rings}, Transactions of the AMS, vol 40 (1936), pages 363--395.
\bibitem{MS} S. MacLane and O.F.G Schilling, {\em Zero-dimensional branches
of rank one on algebraic varieties}, Ann. of Math., vol 40, 3 (1939), pp. 507--520.
\bibitem{Spi} M. Spivakovsky, {\em Valuations in function fields of surfaces}, Amer. J. Math, vol. 112, 1 (1990), pages 107--156.
\bibitem{V0} Vaqui\'e, Michel, {\em Famille admise associ\'ee \`a une valuation de $K[x]$. [Admissible family associated with a valuation of $K[x]$]}, Singularit\'es Franco-Japonaises,  391--428, S\'emin. Congr., 10, Soc. Math. France, Paris, 2005.
\bibitem{V} Vaqui\'e, Michel, {\em Extension d'une valuation. [Extension of a valuation]},  Trans. Amer. Math. Soc.  359  (2007),  no. 7, 3439--3481. (electronic)
\bibitem{V1} Vaqui\'e, Michel, {\em Famille admissible de valuations et d\'efaut d'une extension. [Admissible family of valuations and defect of an extension]},  J. Algebra  311  (2007),  no. 2, 859--876.
\bibitem{V2} Vaqui\'e, Michel, {\em Alg\`ebre gradu\'ee associ\'ee \`a une valuation de $K[x]$. [Graded algebra associated with a valuation of $K[x]$]}  Singularities in geometry and topology 2004,  259--271, Adv. Stud. Pure Math., 46, Math. Soc. Japan, Tokyo, 2007.
\bibitem{V3} Vaqui\'e, Michel, {\em Extensions de valuation et polygone de Newton. [Valuation extensions and Newton polygon]}  Ann. Inst. Fourier (Grenoble)  58  (2008),  no. 7, 2503--2541.

\end{thebibliography}
\end{document}